%% file: AlternatingPorjection-arXiv.tex
\documentclass{article}
\usepackage[top=1in, bottom=1in, left=1in, right=1in]{geometry}

\input{macro}

\providecommand{\keywords}[1]{\textbf{\textit{Keywords:}} #1}

\title{Convergence Analysis of Alternating  Projection Method for Nonconvex Sets}
\author{Zhihui Zhu\thanks{Z. Zhu is with the Center for Imaging Science, Johns Hopkins University, Baltimore, MD, USA (email:zzhu29@jhu.edu)}
~ and  Xiao Li \thanks{Xiao Li is with the Department
	of Electronic Engineering, The Chinese University of Hong Kong, Shatin, NT, Hong Kong (e-mail: xli@ee.cuhk.edu.hk)  }% <-this % stops a space
}

\begin{document}

\maketitle

\begin{abstract}
	Alternating projection method {has been used in a wide range of engineering applications} since it is a gradient-free method (without requiring tuning the step size) and  usually has fast speed of convergence.   In this paper, we formalize two properties of proper, lower semi-continuous and semi-algebraic sets: the three-point property for all possible iterates and the local contraction property that serves as the non-expensiveness property of the projector, but only for the iterates that are close enough to each other. Then by exploiting the geometric properties of the objective function around its critical point, i.e. the Kurdyka-\L{ojasiewicz} (KL) property, we establish a new convergence analysis framework to show that if one set satisfies the three-point property and the other one obeys the local contraction property, the iterates generated by alternating projection method is a convergent sequence and converges to a critical point.  We complete this study by providing convergence rate which depends on the explicit expression of the KL exponent. As a byproduct, we use our new analysis framework to recover the linear convergence rate of alternating projection method onto closed convex sets. To illustrate the power of our new  framework, we provide  new  convergence result for a class of concrete applications: alternating projection method for designing structured tight frames that are widely used in sparse representation, compressed sensing and communication.
\end{abstract}


\section{Introduction}
We consider the problem of finding the minimum Euclidean distance between two sets:
\e
\minimize_{\vx\in\setX,\vy\in\setY} \ g(\vx,\vy) = \|\vx - \vy\|_2^2,
\label{eq:min x-y}\ee
where $\setX$ and $\setY$ are two nonempty closed subsets of $\R^n$ and are possibly nonconvex. A simple but popular approach for solving \eqref{eq:min x-y} is the {\em alternating projection method}  which alternatingly projects the iterates onto the sets $\setX$ and $\setY$:
\e\begin{split}
	&\vx_{k+1} \in \arg\min_{\vx\in\setX} \ g(\vx,\vy_k) = \calP_{\setX}(\vy_k),\\
	&\vy_{k+1} \in \arg\min_{\vy\in\setY} \ g(\vx_{k+1},\vy) = \calP_{\setY}(\vx_{k+1}).
\end{split}\label{eq:apm}\ee
Here for a closed subset $\setV\in\R^n$, $\calP_{\setV}(\cdot)$ represents the orthogonal projection onto $\setV$, that is, $
\calP_{\setV}(\vu):=\argmin_{\vv\in\setV}\|\vv - \vu\|_2^2.$ In case there exist more than one choice for $\vx_{k+1}$ (or $\vy_{k+1}$) in \eqref{eq:apm}, we pick any of them. The alternating projection method for solving \eqref{eq:min x-y} is depicted in Algorithm \ref{alg:APM}.
\begin{algorithm}[htb]
	\caption{Alternating Projection Method}% with considering projection noise}
	\label{alg:APM}
	{\bf Input:}  initialization $y_0\in\setY$, maximal iteration number: $maxIter$ and tolerance: $tol$\\
	{\bf Set:} $k=0$
	\begin{algorithmic}[1]
		\WHILE{$k\leq maxIter$ and $\|\vx_{k} - \vy_{k}\|_2> tol$}
		\STATE $\vx_{k+1} \in \calP_{\setX}(\vy_k)$
		\STATE $\vy_{k+1} \in \calP_{\setY}(\vx_{k+1})$
		\STATE $k\leftarrow k+1$
		\ENDWHILE
	\end{algorithmic}
	%	{\bf Output}:  $\vx_{maxIter}, \vy_{maxIter}$.
\end{algorithm}

Alternating projection method has been widely utilized for solving practical problems provided an efficient way for solving \eqref{eq:apm}.  Compared with gradient-based local search algorithms (such as {projected gradient descent}), the alternating projection method is step-size free and has faster \emph{empirical} convergence speed. Choosing an appropriate step-size is one of the major challenges in gradient-based optimization algorithms. It is easy to implement alternating projection method for many practical applications due to the fact that there is no need to tune the step-size and we only require to solve \eqref{eq:apm} which admits a closed-form solution for many cases. %{On the other hand, like alternating minimization, alternating projection method may fail to converge when applied to very general nonconvex sets \cite{bauschke2014local}.} 

Alternating projection has been widely applied for \emph{convex feasibility problem}; see \cite{bauschke1996projection} for a comprehensive view. In the area of \emph{image restoration}, Youla et al. \cite{youla1982image} estimated the image from its incomplete observation by recursively computing  projections onto closed convex sets and provided theoretical convergence analysis if the underlying ground truth image lies in the intersection of these convex sets; this was further extended in \cite{combettes1997convex} where the revised alternating projection method allows parallel computing and inexact projection at each step.   In \emph{signal processing and inverse problem}, Bauschke et al. \cite{bauschke2002phase} formulated the classical phase retrieval problem into the minimum Euclidean distance framework \eqref{eq:min x-y}, and Byrne \cite{byrne2003unified}  presented a unified treatment for many iterative algorithms in signal processing and inverse problem via an alternating projection perspective. We refer the readers to \cite{escalante2011alternating} and the references therein for many other applications involving alternating projection method.

Although the alternating projection method has been known to work surprisingly well in practice, it remains an active research area to fully understand the theoretical foundation of this phenomenon, especially the convergence behaviors for these methods. {Indeed, for general nonconvex sets, it has been proved that alternating projection method may fail to converge and start to cycle \cite{bauschke2013cluster,bauschke2014local}. Our main interest is the result guaranteeing convergence---the sequence of  iterates is convergent and its limit point satisfies certain optimality conditions---for alternating minimization method.}

\subsection{Previous related work}
The alternating projection method has long history which can be traced back to John Von Neumann \cite{von1950functional}, where the alternating projection between two closed subspaces of a Hilbert space is guaranteed to globally converge to a intersection point of the two subspaces, if they intersect non-trivially. Aronszain \cite{aronszajn1950theory} proved that the rate of  convergence is linear depending on the principal angle between the two subspaces.  Bregman \cite{bregman1965method} extended the alternating projection onto subspaces to projection onto closed convex sets (POCS) with almost similar convergence guarantee. The convergence rate of POCS is known to be linear if the relative interiors of the two convex sets intersect to each other \cite{bauschke1993convergence}.  See \cite{bauschke1996projection} for a comprehensive survey on POCS. Alternating projection method has also been widely utilized when the sets do not intersect. It has been pointed out in \cite{boyd2003alternating} that the sequence generated by the alternating projection method is convergent and converges to a pair of points in $\setX$ and $\setY$ that have Euclidean minimum distance when the two sets are closed convex sets.

Unlike alternating projection between convex sets, the theoretical results for alternating projection method with nonconvex sets are limited.  Tropp et al.~\cite{tropp2005designing} have applied the theorem of Meyer \cite{meyer1976sufficient} to obtain subsequence convergence results for alternating projection method when utilized for a class of nonconvex sets onto which the orthogonal projection is unique. Certain properties of the nonconvex sets have been imposed to obtain stronger convergence results. Lewis et al. \cite{lewis2009local} utilized the notion of  regularity of the intersection between the two sets. In particular, if the two sets have linear regular intersection and at least one set is super-regular at a common point in the intersection area, the alternating projection algorithm is proved to converge  to this common point at a linear rate provided that the algorithm is initialized at a point that is close enough to this common point \cite{lewis2009local}.   Recently, Drusvyatskiy et al. \cite{drusvyatskiy2015transversality} proved that if the two sets intersect transversally at a common point and \Cref{alg:APM} starts with a point close enough to this common point, then the alternating projection algorithm converges linearly to this common point without  the assumption that one set is super-regular at the common point. {Noll et al. \cite{noll2016local} proved convergence of alternating projection method if the two sets intersect each other and one of the sets is H\"{o}lder regular with respect to the other. They also provided convergence rate based on the component of H\"{o}lder regularity. One common assumption in \cite{lewis2009local,drusvyatskiy2015transversality,noll2016local}  is that  the  two nonconvex sets $\setX, \setY$ intersect with each other, i.e., $\setX\cap\setY \neq \emptyset$.  }

{The Douglas-Rachford  algorithm, slightly different to \Cref{alg:APM}, has also been used to solve the feasibility problem  since its introduction in \cite{douglas1956numerical}. Hesse and Luke \cite{hesse2013nonconvex} proved local convergence of the scheme for an affine subspace transversally intersecting another set which can be nonconvex but is required to satisfy a regularity hypothesis called superregularity. Bauschke and Noll \cite{bauschke2014local} provided convergence guarantee of Douglas-Rachford  algorithm when the sets are finite unions of convex sets. They also proved that the Douglas-Rachford  scheme may fail to converge for general nonconvex sets.} 

We finally mention another closely related recent works in proximal algorithms including proximal alternating minimization and projection method~\cite{attouch2010proximal} and proximal alternating linearlized minimization~\cite{bolte2014proximal}. Under the assumption that the objective function satisfies the so-called Kurdyka-\L{ojasiewicz} (KL) inequality \cite{lojasiewicz1963propriete,kurdyka1998gradients}, the convergence of the iterates sequence generated by the proximal alternating algorithms was established in \cite{bolte2007lojasiewicz, bolte2007clarke,attouch2010proximal,bolte2014proximal} for general nonsmooth optimization that is not required to be convex. As pointed out by Bolte et al. \cite{bolte2007lojasiewicz, bolte2007clarke}, the KL inequality is quite universal in the sense that if a function is proper, lower semi-continuous and semi-algebraic or sub-analytical, the function satisfies the KL inequality at any point in its effective domain; see also \cite[Theorem 5.1]{bolte2014proximal}. The KL property is proved to be very useful  for analyzing the convergence behavior of proximal type algorithms solving general nonsmooth and nonconvex problems \cite{attouch2009convergence,attouch2010proximal,attouch2013convergence,bolte2014proximal}. {For example, the KL property has  been utilized to address the  convergence issue of the proximal alternating projection method for general nonconvex sets $\setX$ and $\setY$ in\cite{attouch2010proximal}: with a proximal regularizer, $\vx_k$ and $\vy_k$ are updated respectively by $\vx_{k+1} \in \argmin_{\vx\in\setX}g(\vx,\vy_k) + \eta_{\vx} \|\vx - \vx_{k}\|_2^2 = \calP_{\setX}(\vy_k + \eta_{\vx} \vx_k)$ and $\vy_{k+1} \in \argmin_{\vy\in\setY}g(\vx_{k+1},\vy) + \eta_{\vy} \|\vy - \vy_{k}\|_2^2 = \calP_{\setX}(\vx_{k+1} + \eta_{\vy} \vy_k)$ with $\eta_{\vx},\eta_{\vy}>0$ rather than as in Algorithm~\ref{alg:APM}. The proximal regularizers $\eta_{\vx} \|\vx - \vx_{k}\|_2^2$ and $\eta_{\vy} \|\vy - \vy_{k}\|_2^2$ ensure the convergence of the corresponding algorithm. However, Algorithm~\ref{alg:APM} is widely utilized for practical applications as it is a very simple algorithm and decreases the objective function $g$ in \eqref{eq:min x-y} the most in each step. Thus, we stress out that our main interest is to provide convergence analysis for the alternating projection method, rather than providing new algorithms for solving \eqref{eq:min x-y}. %In particular, the sequence convergence result for Algorithm~\ref{alg:APM} under certain conditions on the sets $\setX$ and $\setY$ (see \Cref{thm:main convergence}) provides theoretical guarantees for the practical utilization of the naive or classical alternating projection method.
}

\subsection{Outline and our contributions}
\label{sec:main contributions}

{
	In this paper, we provide new convergence results for the alternating projection method (i.e, \Cref{alg:APM}) when applied for nonconvex sets $\setX$ and $\setY$. Our main contributions are briefly summarized below immediately followed by detailed descriptions. 
	\begin{itemize}
		
		\item In \Cref{sec:main result}, we prove that the sequence of iterates generated by the alternating projection method is convergent and converges to a critical point of \eqref{eq:min x-y} if the sets satisfy the three-point property and the local contraction property (see \Cref{assum}). 
		\item %To illustrate the power of our convergence analysis framework, we give new convergence results for a class of concrete applications: designing structured tight frames via  \Cref{alg:APM} \cite{tropp2005designing}.  
		As stylized applications of our convergence analysis framework, in Section~\ref{sec:frame design}, we provide sequence convergence that improves upon the previous subsequence convergence result in \cite{tropp2005designing} for designing structured tight frames via alternating projection method.
	\end{itemize}
}

We first emphasize that designing structured tight frames is central to several engineering applications.   For example, equiangular tight frame is a natural choice for sparsely representing signals as it has lower mutual coherence and thus has been extensively utilized in sparse representation and sensing matrix design for compressed sensing system \cite{yaghoobi2009parametric,tropp2005designing,li2013projection,li2015designing,elad2007optimized,tsiligianni2014construction}.  Also designing tight frames with prescribed column norm is crucial for direct sequence-code division multiple access (DS-CDMA)  in communication \cite{tropp2005designing} as it is directly related to the construction of the optimal signature sequences.

The underpinning fact from which the new result is established is the utilization of the three-point property to guarantee the asymptotic regular property of the sequence of iterates  in terms of \emph{one variable} and the local contraction property to ensure similar asymptotic regular property of the sequence of iterates in terms of the other variable. The sequence convergence property is then obtained by exploiting the KL property of the objective function. We complete this result by the study of the convergence rate which depends on the explicit expression for the KL exponent characterizing the geometrical properties of the problem around its critical points. {The standard convergence analysis framework based on KL inequality is established in \cite{absil2005convergence,attouch2009convergence,attouch2010proximal,bolte2014proximal}. According to the comments above \cite[Theorem 1]{attouch2009convergence}, this framework originally appears in \cite{lojasiewicz1963propriete}. We will utilize this analysis framework with slight variation  to show the convergence of the method of alternating projection. Our analysis framework and its difference with the standard one are described in \Cref{sec:proof highlight}.   }

{Unlike the convergence results in \cite{lewis2009local,drusvyatskiy2015transversality,noll2016local} that require the two sets intersect each other and an initialization that is within (or near) the intersection area, our result can be applied to any two sets that have an empty intersection.} Checking if the two sets intersect each other is non-trivial; it is even harder to find such a proper initialization that is close enough to the intersection area. Also, as the examples given in Section \ref{sec:frame design}, it is common that the two sets do not intersect each other and the goal is to find a pair of points that have minimum distance.

As the subspaces and closed convex sets automatically satisfy the three-point property and the local contraction property, our results cover the sequence convergence result (with linear rate convergence) for alternating projection onto subspaces and closed convex sets~\cite{boyd2003alternating}. However, our proof technique differs to \cite{boyd2003alternating} in that we exploit the geometric properties of the objective function around its critical points, i.e, the KL property which enables us to apply our results to general closed nonconvex, semi-algebraic sets.% that obey the three-point property and the local contraction property. 

\section{Convergence analysis for the alternating projection method}
\label{sec:main result}

We start with some improtant definitions.
\begin{defi}\cite{bolte2014proximal}Let $h:\R^d\rightarrow (-\infty,\infty]$ be a proper and lower semi-continuous function, whose domain is defined as $
	\domain h:=\left\{\vu\in\R^d:h(\vu)<\infty\right\}.$
	The Fr\'{e}chet subdifferential $\widehat\partial h$ of $h$ at $\vu\in \domain h$ is defined by
	\[
	\widehat\partial h(\vu) = \left\{\vz:\lim_{\vv\rightarrow \vu,\vv\neq  \vu}\inf \frac{h(\vv) - h(\vu) - \langle \vz, \vv - \vu\rangle}{\|\vu - \vv\|}\geq 0\right\}.
	\]
	$\widehat\partial h(\vu) = \emptyset$ if $\vu\notin \domain h$. The limiting subdifferential $\partial h(\vu)$  is defined as follows
	\[
	\partial h(\vu) = \left\{
	\vz:\exists \vu_k\rightarrow \vu, h(\vu_k)\rightarrow h(\vu), \vz_k\in \widehat\partial h(\vu_k)\rightarrow \vz
	\right\}.
	\]
\end{defi}
We say $\overline{\vu}$ a limiting critical point of $h$ if it satisfies the first-order optimality condition ${\bf 0} \in \partial h(\overline{\vu})$. Throughout the paper, when it is clear from the context, we omit the word ``limiting'' and just call $\partial h(\vu)$ and $\overline\vu$ as the subdifferential and critical point of $h$, respectively. The following KL property characterizes the local geometric properties of the objective function around its critical points and is proved to be pretty useful for convergence analysis \cite{attouch2009convergence,attouch2010proximal,attouch2013convergence,bolte2014proximal}.

\begin{defi}\cite{attouch2009convergence}\label{def:KL}
	A proper semi-continuous function $h(\vu)$ is said to satisfy the Kurdyka-\L{ojasiewicz} (KL) property, if for any critical point $\overline{\vu}$ of $h(\vu)$, there exist $\delta>0,~\eta>0,~\theta\in[0,1),~C_1>0$ such that (where $\theta$ is often referred to as the KL exponent)
	for all
	\[
	\vu \in B(\overline{\vu}, \delta) \cap \{\vu: h(\overline \vu) <  h( \vu) < h(\overline \vu) + \eta\}
	\]
	we have
	\[
	\left|h(\vu) - h(\overline{\vu})\right|^{\theta} \leq C_1 \dist(0, \partial h(\vu)).
	\]
\end{defi}

We then give out the main assumption we made in this paper to show the convergence of alternating projection method.

%\begin{assum}\label{assum}
%	Let $\setX$ and $\setY$ be  two closed semi-algebraic sets, and let $\{(\vx_k,\vy_k)\}$ be the sequence of iterates generated by the alternating projection method (i.e., Algorithm \ref{alg:APM}). Assume the sequence $\{(\vx_k,\vy_k)\}$ is bounded and the sets $\setX$ and $\setY$ obey the following properties: \notexli{we may change the notion three-point property to partial sufficient decrease}.
%	\begin{enumerate}[(i)]
%		\item three-point property of $\setY$: there exists a nonnegative function $\delta_\alpha:\setY\times\setY\rightarrow \R$ with $\alpha>0$ such that $\delta_\alpha(\vy_k,\vy_{k-1})\geq \alpha \|\vy_k-\vy_{k-1}\|_2^2$ and
%		\e
%		\delta_\alpha(\vy_{k-1},\vy_k) + g(\vx_k,\vy_k)\leq g(\vx_k,\vy_{k-1}), \ \forall \ k\geq 1;
%		\label{eq:3point property}\ee
%		\item local contraction property of $\setX$: there exist $\eps>0$ and $\beta>0$ such that when $\|\vy_k-\vy_{k-1}\|_2\leq \eps$, we have
%		\e \label{eq: local contraction}
%		\|\vx_{k+1} - \vx_{k}\| = \| \calP_{\setX} (\vy_k) - \calP_{\setX}(\vy_{k-1}) \|_2 \leq \beta \|\vy_k - \vy_{k-1}\|_2.
%		\ee
%	\end{enumerate}
%\end{assum}

{
	
	\begin{assum}\label{assum}
		Let $\setX$ and $\setY$ be any two closed semi-algebraic sets, and let $\{(\vx_k,\vy_k)\}$ be the sequence of iterates generated by the alternating projection method (i.e., Algorithm \ref{alg:APM}). Assume that the sequence $\{(\vx_k,\vy_k)\}$ is bounded and there exist subsets $\overline\setX\subset \setX$ and $\overline\setY\subset \setY$ and $k_0\in\N$ such that $\vx_k\in \overline\setX$ and $\vy_k\in\overline\setY$ for all $k\geq k_0$. Furthermore, we assume that the sets $\setX, \setY$ and subsets $\overline\setX, \overline\setY$ obey the following properties:
		\begin{enumerate}[(i)]
			\item three-point property: there exists a nonnegative function $\delta_\alpha:\setY\times\setY\rightarrow \R$ with $\alpha>0$ such that $(i)$ for all $\vy,\vy'\in \setY$ we have $\delta_\alpha(\vy,\vy')\geq \alpha \|\vy-\vy'\|_2$  and $(ii)$  for all $\vy\in \overline\setY, \widetilde\vx\in\overline\setX, \widetilde\vy \in \argmin_{\vy'\in\setY} g(\widetilde\vx,\vy')$, we have 
			\e
			\delta_\alpha(\vy,\widetilde\vy) + g(\widetilde\vx,\widetilde\vy)\leq g(\widetilde\vx,\vy);
			\label{eq:3point property}\ee
			\item local contraction property of $\setX$ with respect to $\overline\setY$:  there exist $\eps>0$ and $\beta>0$ such that
			\e \label{eq: local contraction}
			\| \calP_{\setX} (\widetilde\vy) - \calP_{\setX}(\vy) \|_2 \leq \beta \|\widetilde\vy - \vy\|_2, \ \forall \ \vy,\widetilde\vy\in \overline\setY, \|\widetilde\vy-\vy\|_2\leq \eps.
			\ee
		\end{enumerate}
	\end{assum}
}

{
	\begin{remark}
		%	One may notice that aside from properties in \eqref{eq:3point property} and \eqref{eq: local contraction} about the sets, \Cref{assum} also requires $\vx_k\in \overline\setX$ and $\vy_k\in\overline\setY$. 
		We first explain the reason that we impose the assumption $\vx_k\in \overline\setX$ and $\vy_k\in\overline\setY$. The reason is that for some nonconvex sets (like the unit sphere (i.e., the boundary of the unit ball) in the next remark), the three-point property \eqref{eq:3point property} and the local contraction property \eqref{eq: local contraction} may only hold with respect to their subsets, but not the entire sets. Thus, with the assumption $\vx_k\in \overline\setX$ and $\vy_k\in\overline\setY$, we only require the properties \eqref{eq:3point property} and  \eqref{eq: local contraction} hold for the points in the subsets. As will been seen in \Cref{sec:frame design}, this allows us to provide improved convergence guarantee for alternating projection method in designing structured tight frames. 
\end{remark}}
\begin{remark}
	This three-point property~\eqref{eq:3point property} along with a so-called four-point property has been widely utilized for proving the convergence of the sequence $\{g(\vx_k,\vy_k)\}$ (rather than the iterates $\{(\vx_k,\vy_k)\}$) generated by alternating minimization~\cite{csisz1984information,niesen2009adaptive}. As we consider the convergence of the iterates, the $\delta_\alpha$ function in \eqref{eq:3point property} is slightly stronger than the one in~\cite{csisz1984information,niesen2009adaptive}, where the function $\delta_\alpha$ is only required to be positive, i.e, $\delta_\alpha(\vy,\vy')>0$ for all $\vy,\vy'\in\setY$ and $\vy\neq \vy'$. {In particular, as will be seen soon in \Cref{sec:proof highlight}, the requirement that $\delta_\alpha(\vy,\vy')\geq \alpha \|\vy-\vy'\|_2$ provides (partial) sufficient decrease property which is central to the following convergence analysis. Similar three-point property has also been used in \cite{noll2016local} for convergence analysis where the authors proved that such property holds for general nonconvex sets that intersect each other and one of the sets is H\"{o}lder regular with respect to the other.
		Before going to the details of convergence analysis, we first give out examples satisfy the three-point property and the local contraction property.  
	}
	
	We note that the three-point property~\eqref{eq:3point property} mostly characterizes a certain property regarding the set $\setY$. A typical example satisfying this three-point property~\eqref{eq:3point property} is a convex and closed set $\setY$ which obeys \eqref{eq:3point property} for any $\vy\in\setY,\widetilde\vx\in\R^n, \widetilde\vy \in \argmin_{\vy'\in\setY} g(\widetilde\vx,\vy')$ with $\delta_\alpha(\vy,\widetilde\vy) = \|\vy -\widetilde\vy\|_2^2$ since 
	\e\begin{split}
		&g(\widetilde\vx,\vy) - g(\widetilde\vx,\widetilde\vy)  =  \| \widetilde\vx- \vy \|_2^2 - \|\widetilde\vx - \widetilde\vy\|_2^2\\&= \|\widetilde \vx - \widetilde\vy + \widetilde\vy - \vy \|_2^2 - \|\widetilde\vx - \widetilde\vy\|_2^2\\
		& =\|\widetilde\vy - \vy \|_2^2 + 2\langle \widetilde\vx - \widetilde\vy,  \widetilde\vy - \vy  \rangle  \geq \|\widetilde\vy - \vy \|_2^2,
	\end{split}\label{eq:3point convex set}\ee
	where the last inequality follows from the fact that $\setY$ is a closed convex set such that
	\e
	\langle \widetilde\vx - \widetilde\vy,  \widetilde\vy - \vy'  \rangle \geq 0,  \ \forall \  \vy'\in \setY.
	\label{eq:projection convex propety}\ee
	Another example is the unit sphere $\setY=\left\{\vy\in\R^n:\|\vy\|_2 = 1\right\}$ which satisfies \eqref{eq:3point property} for any $\widetilde\vx$ that is not zero.  In particular, for any $\vx\in\R^n$, its projection onto $\setY$ is defined as
	\e
	\calP_{\setY}(\vx) = \left\{\begin{matrix}\frac{\vx}{\|\vx\|_2}, & \vx\neq \vzero,\\ \vu, & \vx = \vzero,\end{matrix}\right.
	\label{eq:proj unit sphere}\ee
	where $\vu$ represents an arbitrary unit vector. Now by defining $\widetilde\vy \in \calP_{\setY}(\widetilde\vx)$, for any $\vy\in\setY$, we have
	\e\begin{split}
		&	\| \widetilde\vx - \vy\|_2^2 - \| \widetilde\vx - \widetilde\vy\|_2^2 = 2\widetilde\vx^\T\widetilde\vy -2\widetilde\vx^\T\vy  = \|\widetilde\vx\|_2 \left(2\widetilde\vy^\T\widetilde\vy - 2 \widetilde\vy^\T \vy \right)\\
		& = \|\widetilde\vx\|_2 \left(\widetilde\vy^\T\widetilde\vy - 2 \widetilde\vy^\T \vy + \|\widetilde\vy\|_2^2\right)=\|\widetilde\vx\|_2\|\widetilde\vy - \vy\|_2^2,
	\end{split}\label{eq:3point unit sphere}\ee
	where the first line utilizes $\widetilde\vy \in \calP_{\setY}(\widetilde\vx)$ and the second line follows from $\|\widetilde\vy\|_2 = \|\vy\|_2 = 1$. It is clear from \eqref{eq:3point unit sphere} that the three-point property \eqref{eq:3point property} holds for all $\vx$ that is away from zero. %\revise{Noll and Rondepierre \cite{noll2016local} proved that the three-point property holds for general nonconvex sets that intersect each other and one of the sets is H\"{o}lder regular with respect to the other.}
	
	%With this example, we stress that the three-point property \eqref{eq:3point property} of $\setY$ is only required to hold for all possible iterates $\vx_k$ rather than for any $\vx\in\setX$.
	
\end{remark}

\begin{remark}
	The local contraction property in \eqref{eq: local contraction} is mild and it basically requires the projections of $\widetilde\vy$ and $\vy$ onto $\setX$ are not far away when $\vy$ is close enough to $\widetilde\vy$. This property is expected to hold if we want to guarantee the convergence of the alternating projection method. A typical class satisfying this  local contraction property  \eqref{eq: local contraction} is a closed convex set $\setX$ with $\beta = 1$ and $\eps$ be arbitrary positive number in \eqref{eq: local contraction}:
	\e
	\| \calP_{\setX} (\widetilde\vy) - \calP_{\setX}(\vy) \|_2 \leq \| \widetilde\vy - \vy\|_2
	\label{eq:local conctrac convex set}\ee
	for arbitrary $\widetilde\vy, \vy \in \R^n$ (not only the algorithm trajectory). \eqref{eq:local conctrac convex set} is also known as the non-expensiveness property of orthogonal projection onto closed convex sets. {Another example satisfying this local contraction property  \eqref{eq: local contraction} is the unit sphere $\setX =\left\{\vx\in\R^n:\|\vx\|_2 = 1\right\}$ for nonzero vectors. Specifically, for any constant $c>0$, denote by $\overline\setY = \{\vy\in\R^n: \|\vy\|_2 \ge c\}$ the set of vectors that excludes the origin. Then, for any $\vy,\wt \vy\in \overline\setY$, we have
		\begin{align*}
		\norm{\calP_{\setX} (\widetilde\vy) - \calP_{\setX}(\vy)}{2} &= \norm{\frac{\vy}{\norm{\vy}{2}} - \frac{\wt\vy}{\norm{\wt\vy}{2}} }{2}	 = \frac{1}{{\norm{\vy}{2} \norm{\wt\vy}{2}}} \Big\|\norm{\wt\vy}{2}\vy -  \norm{\vy}{2}\wt \vy  \Big\|_2\\
		& = \frac{1}{{\norm{\vy}{2} \norm{\wt\vy}{2}}}\Big\|\norm{\wt\vy}{2}\vy - \norm{\wt\vy}{2}\wt\vy + \norm{\wt\vy}{2}\wt\vy - \norm{\vy}{2}\wt \vy\Big\|_2\\
		& \le  \frac{1}{{\norm{\vy}{2} \norm{\wt\vy}{2}}}\parans{\norm{\wt\vy}{2} \norm{\vy - \wt\vy}{2} + \norm{\wt\vy}{2} \big | \norm{\wt\vy}{2} - \norm{\vy}{2}\big | }\\
		& \le \frac{2}{\norm{\vy}{2}} \norm{\vy - \wt\vy}{2} \le \frac{2}{c} \norm{\vy - \wt\vy}{2}.
		\end{align*}
		Thus, the local contraction property \eqref{eq: local contraction} is satisfied with $\beta = \frac{2}{c}$ and arbitrary positive $\epsilon$.
		The other examples including the set of tight frames are presented in Section~\ref{sec:frame design}.
	}

\end{remark}

{
	\subsection{Consequence of \Cref{assum} and proof highlight}
	\label{sec:proof highlight}
	Let $\{(\vx_k,\vy_k)\}$ be the sequence of iterates generated by the alternating projection method. As a consequence of \Cref{assum} and particularly \eqref{eq:3point property}, we have
	\e
	\delta_\alpha(\vy_{k-1},\vy_k) + g(\vx_k,\vy_k)\leq g(\vx_k,\vy_{k-1}), \ \forall \ k\geq k_0+1,
	\label{eq:3point property intro}\ee
	which provides sufficient decrease property with respect to $\{\vy_k\}$. To give out the proof highlight, we transfer the constrained problem into the following equivalent form without any constraints:
	\e
	f(\vx,\vy) = g(\vx,\vy) + \delta_{\setX}(\vx) + \delta_{\setY}(\vy),
	\label{eq:unconstraind}\ee
	where $\delta_{\setX}(\vx) = \left\{ \begin{matrix} 0, & \vx\in\setX \\ \infty, & \vx\notin \setX\end{matrix}\right\}$ (and $\delta_{\setY}$) is the indicator function of the set $\setX$ (and $\setY$). We now give some insights into our proof strategy.
	\begin{enumerate}[(i)]
		\item (partial) sufficient decrease property: by using \eqref{eq:3point property intro} we obtain that
		\e
		g(\vx_{k-1},\vy_{k-1}) - g(\vx_k,\vy_{k}) \geq \alpha\|\vy_{k-1} - \vy_k\|_2^2, \ \forall k \ge k_0+1
		\label{eq:sufficient decreae intro}\ee
		which guarantees the  asymptotic regular property of $\{\vy_k\}$, i.e., $\lim_{k\rightarrow\infty}\|\vy_k - \vy_{k-1}\|_2 =0$. This together with the local contraction property \eqref{eq: local contraction} gives the  asymptotic regular property of $\{\vx_k\}$;
		\item safeguard property: find a positive constant $c>0$ and construct $\vd_k \in \partial f(\vx_k,\vy_k)$ such that
		\[
		\|\vd_k\|_2 \leq c \|\vy_{k} - \vy_{k-1}\|_2;
		\]
		\item KL property: show that the sequence $\{(\vx_k,\vy_k)\}$ is a {\em Cauchy sequence}.
	\end{enumerate}
	We note that the first two requirements are slightly different from the standard ones that are shared by most descent algorithms~\cite{attouch2009convergence,attouch2010proximal,bolte2014proximal}. We use the first requirement as an example to illustrate the difference. As pointed out in \cite{attouch2009convergence,attouch2010proximal,bolte2014proximal}, the standard sufficient decrease property has the form
	\e
	g(\vx_{k-1},\vy_{k-1}) - g(\vx_k,\vy_{k}) \geq \alpha\left(\|\vy_{k-1} - \vy_k\|_2^2 + \|\vx_{k-1} - \vx_k\|_2^2\right),
	\label{eq:sufficient decreae intro classical}\ee
	which is stronger than \eqref{eq:sufficient decreae intro}. The partial sufficient decrease property in \eqref{eq:sufficient decreae intro} that depends on the iterates gap of only one variable provides us the freedom to put different requirements on the two sets. %Typical examples for $\setY$ satisfying the three-point property (and hence \eqref{eq:sufficient decreae intro}) include convex sets and unit sphere ({i.e., the boundary of unit ball}). The assumption of the local contraction property (see \eqref{eq: local contraction}) on the set $\setX$ is  mild as it basically requires that $\| \calP_{\setX} (\vy_k) - \calP_{\setX}(\vy_{k-1}) \|_2$ is small when $\|\vy_{k} - \vy_{k-1}\|_2$ converges to 0. On the other hand, the classical sufficient decrease property \eqref{eq:sufficient decreae intro classical} depends on the iterates gap of both variables and thus adds similar requirement on both sets.
}

\subsection{Convergence to a critical value}

%We transfer the constrained problem into the following equivalent form without any constraints: \e f(\vx,\vy) = g(\vx,\vy) + \delta_{\setX}(\vx) + \delta_{\setY}(\vy), \label{eq:unconstraind}\ee where $\delta_{\setX}(\vx) = \left\{ \begin{matrix} 0, & \vx\in\setX \\ \infty, & \vx\notin \setX\end{matrix}\right.$ (and $\delta_{\setY}$) is the indicator function of the set $\setX$ (and $\setY$).

To simplify the notation, we stack $\vx$ and $\vy$ into one variable as $\vz = (\vx,\vy)$.  Let $\{(\vx_k,\vy_k)\}$ be the sequence of iterates generated by the alternating projection method. %\revise{Since our convergence analysis mainly characterizes the behavior of the iterate $\vz_k$ when $k$ is large enough, without loss of generality, we simply assume $k_0 = 0$ in \Cref{assum} throughout this section. Thus, \Cref{assum} implies that $\vx_k \in \overline\setX, \vy_k \in \overline\setY$ and there exists an $\alpha>0$ such that
%	\e \alpha\|\vy_{k-1} - \vy_k\|_2 + g(\vx_k,\vy_k)\leq g(\vx_k,\vy_{k-1}), \ \forall \ k\geq 1. \label{eq:3point property}\ee }
With \Cref{assum}, we begin by showing the convergence of $\{f(\vz_k)\}$ and that the sequence $\{\vz_k\}$ is regular (i.e., $\lim_{k\rightarrow\infty} \|\vz_k - \vz_{k-1}\|_2 = 0$) in the following result.

\begin{lem}\label{lem:sufficient decrese} Under \Cref{assum}, we have the following assertions.
	\begin{enumerate}[(i)]
		\item The sequence $\{f(\vz_{k})\}$ is monotonically decreasing and convergent.
		\item We have
		\begin{align}
		&f(\vz_{k-1}) - f(\vz_k) \geq \alpha\|\vy_{k-1} - \vy_k\|_2^2, \ \forall \ k\ge k_0+1,
		\label{eq:sufficient decrease}\\
		&\lim_{k\rightarrow\infty} \|\vy_k - \vy_{k-1}\|_2 = 0,~~	\lim_{k\rightarrow\infty} \|\vx_k - \vx_{k-1}\|_2 = 0,
		\label{eq:y difference converges}\end{align}
		{where $\alpha >0$ is defined in the three-point property \eqref{eq:3point property}.}

		%		\item The sequence $\{\vz_k= (\vx_k,\vy_k)\}$ is bounded.		
		\item Denote by
		$\vd_k = (2(\vy_{k-1} - \vy_{k}), \vzero)$ for all $k\geq 1$. Then
		\e
		\vd_k \in \partial f(\vx_{k},\vy_k).
		\label{eq:subdifferential}\ee
	\end{enumerate}
\end{lem}

\begin{proof}[Proof of \Cref{lem:sufficient decrese}] Show $(i)$: {Since by \Cref{assum} that $\vx_0\in\setX,\vy_0\in\setY$ and both $x_0$ and $y_0$ are bounded, we have $0\le f(\vx_0,\vy_0) < \infty$.} Since  $\vx_k \in \argmin_{\vx\in\setX}\|\vx - \vy_{k-1}\|_2^2$, we have
	$f(\vx_k,\vy_{k-1})\leq f(\vz_{k-1})$. Similarly, we have $f(\vx_k,\vy_{k}) \le f(\vx_k,\vy_{k-1})\leq f(\vz_{k-1})$. Hence we have $
	f(\vz_{k-1}) \geq f(\vz_k), \forall ~ k\geq 1,$
	which together with the fact that $\inf_{\vz} f(\vz) \geq 0$ gives that the sequence $\{f(\vz_{k})\}$ is monotonically decreasing and lower bounded, hence convergent.

	Show $(ii)$: due to the assumption $\vx_k\in \overline\setX, \vy_k\in\overline\setY,\forall k\ge k_0$, it follows from \eqref{eq:3point property} that
	\e\begin{split}
		f(\vx_k,\vy_{k-1}) - f(\vz_k) &  \geq \alpha\|\vy_{k-1} - \vy_k\|_2^2, \ \forall k\ge k_0+1,
	\end{split}\label{eq:sufficient decrease y 1}\ee
	which together with $f(\vx_k,\vy_{k-1})\leq f(\vz_{k-1})$ gives \eqref{eq:sufficient decrease}.
	Repeating \eqref{eq:sufficient decrease} for all $k$ and summing them up, we have
	\[
	\sum_{k=k_0 +1}^\infty \|\vy_{k-1} - \vy_k\|_2^2 \leq \frac{1}{\alpha}\sum_{k=k_0 + 1}^\infty [f(\vz_{k-1}) - f(\vz_k)] \leq \frac{1}{\alpha} f(\vx_0,\vy_0),
	\]
	which immediately implies $
	\lim_{k\rightarrow\infty} \|\vy_k - \vy_{k-1}\|_2 = 0.$
	The above equation implies that	for any $\epsilon >0$, there exists $k_1$ such that $\|\vy_{k} - \vy_{k-1}\|_2 \leq \epsilon,~\forall ~ k\geq k_1$. Picking $\epsilon$ such that \eqref{eq: local contraction} holds, we have $
	\|\vx_{k+1} - \vx_{k}\|_2 \leq \beta \|\vy_{k+1} - \vy_{k}\|_2$
	for all $k\geq k_1$. Letting $k\rightarrow \infty$, we conclude $
	\lim_{k\rightarrow\infty} \|\vx_k - \vx_{k-1}\|_2 = 0.$

	%	Show $(iii)$: from the statement $(ii)$, we have
	% \[
	%	 \infty>\|\vx_0- \vy_0 \|_2^2 = f(\vx_0,\vy_0)\geq f(\vx_{k},\vy_{k}) = \|\vx_k- \vy_k \|_2^2 ~~~~\forall ~ k\geq 1
	%	\]
	%	which implies both $\{\vx_{k}\}$ and $\{\vy_{k}\}$ are bounded.

	Show $(iii)$: By the definition of $\vx_k$, $\vzero$ must lie in the subdifferential at $\vx_k$ of the function $\vx \longmapsto f(\vx,\vy_{k-1})$. Hence
	\e
	\vzero \in 2(\vx_k - \vy_{k-1}) + \partial\delta_{\setX}(\vx_k).
	\label{eq:partial x}\ee
	And similarly $
	\vzero \in \partial_{\vy} f(\vx_{k},\vy_k).$
	Noting that $
	\partial_{\vx} f(\vx_{k},\vy_k) = 2(\vx_k - \vy_{k}) +\partial\delta_{\setX}(\vx_k),$
	which together with \eqref{eq:partial x} gives
	\begin{align*}
	2(\vx_k - \vy_{k}) - 2(\vx_k - \vy_{k-1}) = 2(\vy_{k-1} - \vy_{k})\in\partial_{\vx} f(\vx_{k},\vy_k).
	\end{align*}
	{Note that $g$ in \eqref{eq:unconstraind} is continuously differentiable, we have $\partial f(\vx,\vy) = \big(\nabla_{\vx} g(\vx,\vy) + \partial\delta_{\setX}(\vx), \nabla_{\vy} g(\vx,\vy) + \partial\delta_{\setY}(\vy) \big) = \big(\partial_{\vx} f(\vx,\vy), \partial_{\vy} f(\vx,\vy) \big)$ \cite[Proposition 3]{attouch2010proximal}.}  Thus,  we have $
	(2(\vy_{k-1} - \vy_{k}),\vzero)\in \partial f(\vx_k,\vy_k).$
	This completes the proof for \Cref{lem:sufficient decrese}.
\end{proof}

\Cref{lem:sufficient decrese} ensures a sufficient decrease of the objective function after one step update of $\vx$ and $\vy$. However, we note that the sufficient decrease guaranteed by \eqref{eq:sufficient decrease} is slightly different than the classical one in convergence analysis (like in \cite{attouch2010proximal}) where $f(\vz_{k-1}) - f(\vz_k) \geq c\|\vz_{k-1} - \vz_k\|_2^2$ for some $c>0$ is required.

Let $\calL(\vz_0)$ denote the set of limit points of $\{\vz_k\}$, i.e.,
\[
\calL(\vz_0) = \left\{\overline\vz\in\R^{n}\times \R^n: \exists\ \{k_m\}_{m\in\N}, \ \textup{such that} \lim_{m\rightarrow \infty}\vz_{k_m} = \overline \vz \right\}.
\]
The following result establishes several properties of the limit points set $\calL(\vz_0)$.
\begin{lem}\label{lem:safeguard}
	Under \Cref{assum}, $\calL(\vz_0)$ obeys the following properties.
	\begin{enumerate}[(i)]
		\item $\calL(\vz_0)$ is a nonempty compact connected set and the iterates  $\{\vz_k\}$ satisfies
		\e
		\lim_{k\rightarrow \infty} \dist(\vz_k,\calL(\vz_0)) = 0.
		\label{eq:dist to limit point 0}\ee
		
		\item The objective function $f$ is finite and constant on $\calL(\vz_0)$ and
		\e
		\lim_{k\rightarrow \infty} f(\vz_k) = f(\vz^\star),\ \forall \ \vz^\star\in\calL(\vz_0).
		\label{eq:limit f}\ee
		
		\item Any $\vz^\star\in\calL(\vz_0)$ is a critical point of \eqref{eq:unconstraind}.
	\end{enumerate}
\end{lem}

\begin{proof}[Proof of \Cref{lem:safeguard}]
	Show $(i)$: It is clear that $\calL(\vz_0)$ has at least one convergent subsequence since by assumption the sequence $\{(\vx_k,\vy_k)\}$ is bounded. Also noting that $\calL(\vz_0) = \bigcap_{l\in \setN}\overline{\bigcup_{k\geq l}\{\vz_{k}\}}$ and $\{\vz_{k}\}$ lies in a closed set and it is bounded, the set $ \overline{\bigcup_{k\geq l}\{\vz_{k}\}}$ is compact for any $l\geq 0$. Thus, we conclude that $\calL(\vz_0)$ is compact by interpreting it as the intersection of compact sets. {The connectedness of $\calL(\vz_0)$ and \eqref{eq:dist to limit point 0} follow from \cite[Lemma 3.5]{bolte2014proximal} by utilizing the property $\lim_{k\rightarrow\infty} \|\vx_k - \vx_{k-1}\|_2 + \|\vy_k - \vy_{k-1}\|_2= 0$.}
	
	Show $(ii)$: we extract an arbitrary convergent subsequence $\{\vz_{k_m}\}_m$ from $\{\vz_{k}\}$ with limit $\vz^\star =(\vx^\star,\vy^\star)$. Since we have $\vx_{k_m}\in \setX$ and $\vy_{k_m}\in \setY$ for all $k_m\geq k_0$, it follows from the closedness of $\setX$ and $\setY$ that $\vx^\star\in \setX , ~\vy^\star\in \setY$ and
	\[
	\delta_{\setX}(\vx^\star)= \delta_{\setX}(\vx_{k_m}) = 0, ~~\delta_{\setY}(\vy^\star)= \delta_{\setY}(\vy_{k_m}) = 0,~~~~\forall~m\geq 1,
	\]
	which together with the fact that $g$ is a continuous function gives
	\[
	\lim_{m\rightarrow \infty} f(\vz_{k_m}) =\lim_{m\rightarrow \infty} g(\vz_{k_m}) + \delta_{\setX}(\vx_{k_m}) +\delta_{\setY}(\vy_{k_m}) = f(\vz^\star).
	\]
	Now utilizing the statement $(ii)$ in \Cref{lem:sufficient decrese} that the sequence $\{f(\vz_{k})\}$ is convergent, we have $
	f(\vz^\star)= \lim_{m\rightarrow \infty} f(\vz_{k_m}) = \lim_{k\rightarrow \infty} f(\vz_{k}).$ Thus the objective function $f$ is constant on $\calL(\vz_0)$ since $\vz^\star$ is the limit point of any convergent subsequence.

	Show $(iii)$: 	It follows from \eqref{eq:y difference converges} and \eqref{eq:subdifferential} that $\vd_k\in \partial f(\vx_k,\vy_k)$ and $
	\lim\limits_{k\rightarrow\infty} \vd_k = {\bf 0}.$
	Now for any convergent subsequence $\{\vz_{k_m}\}_m$ with limit $\vz^\star$, we have that $(\vz_{k_m}, \vd_{k_m})$ belongs to the graph of $\partial f$ and $(\vz_{k_m}, \vd_{k_m}) \rightarrow (\vz^\star, {\bf 0})$. By invoking \eqref{eq:limit f} and the definition of $\partial f$, we immediately conclude that $(\vz^\star, {\bf 0})$  belongs to the graph of $\partial f$, hence $
	\vzero\in \partial f(\vz^\star),$
	which implies that $\vz^\star$ is a critical point for \eqref{eq:min x-y}.
\end{proof}

\subsection{Convergence to a critical point}
The following result establishes that $f$ obeys the KL property at $\calL(\vz_0)$.
\begin{lem}\label{lem:KL}
	There exist uniform constants $C > 0, \delta >0$ and $\theta \in [0,1)$ such
	that
	\e
	\left|f(\vz) - f(\vz^\star)\right|^{\theta} \leq C \dist(0, \partial f(\vz))
	\label{eq:KL for f}\ee
	for any $\vz^\star\in \calL(\vz_0)$ and $\vz\in\R^{2n}$ with $\dist(\vz, \calL(\vz_0))\leq \delta$.
\end{lem}

\begin{proof}[Proof of \Cref{lem:KL}]
	Under the semi-algebraic assumption of sets $\setX$ and $\setY$, we immediately conclude that the indicator functions $\delta_{\setX}(\vx)$ and $\delta_{\setY}(\vy)$ are semi-algebraic. We then have $f$ satisfies the KL property at any point in its effective domain, since it is lower semi-continuous and semi-algebraic \cite{bolte2007clarke}. 	 The remaining proof follows from \Cref{lem:sufficient decrese} in \cite{attouch2009convergence} and \Cref{lem:safeguard}.
\end{proof}

\begin{thm}\label{thm:main convergence}
	Under  \Cref{assum}, the sequence $\{(\vx_k,\vy_k)\}$ is convergent and converges to a critical point of \eqref{eq:unconstraind}.
\end{thm}

\begin{proof}[Proof of \Cref{thm:main convergence}]
	By invoking \eqref{eq:dist to limit point 0} and \eqref{eq:limit f},  there exists a finite integer $k_2\ge k_0$ such that $\text{dist}(\vz_k, \calL(\vz_0))\leq \delta$ and $f(\vz^\star)< f(\vz_k)< f(\vz^\star) + \eta$ with any $\eta>0$ for all $k\geq k_2$. Now from the concavity of the function $t^{1-\theta}$ with domain $t>0$, we have
	\[
	\left(f(\vz_{k+1}) - f(\vz^\star )\right)^{1-\theta} \leq \left(f(\vz_{k}) - f(\vz^\star)\right)^{1-\theta} + (1-\theta)\frac{f(\vz_{k+1}) - f(\vz_{k})}{\left(f(\vz_k) - f(\vz^\star)\right)^\theta}.
	\]
	Thus, for all $k\geq k_2$,
	\begin{align}\label{convexity}
	&\left(f(\vz_{k}) - f(\vz^\star)\right)^{1-\theta} - \left(f(\vz_{k+1}) - f(\vz^\star)\right)^{1-\theta} \nonumber\\
	&\geq (1-\theta)\frac{f(\vz_{k}) - f(\vz_{k+1})}{|f(\vz_k) - f(\vz^\star)|^\theta}\nonumber\\
	& \geq (1-\theta)\frac{\|\vy_k - \vy_{k+1}\|_2^2}{C \dist(0,\partial f(\vx_k,\vy_k))} \nonumber\\
	& \geq \frac{1-\theta}{C}\frac{\|\vy_k - \vy_{k+1}\|_2^2}{2\|\vy_{k-1} - \vy_k\|_2} \nonumber\\
	& = \frac{1-\theta}{2C}\left(\frac{\|\vy_k - \vy_{k+1}\|_2^2}{\|\vy_{k-1} - \vy_k\|_2} + \|\vy_{k-1} - \vy_k\|_2 - \|\vy_{k-1} - \vy_k\|_2\right) \nonumber\\
	& \geq \frac{1-\theta}{2C}\left(2\|\vy_k - \vy_{k+1}\|_2 - \|\vy_{k-1} - \vy_k\|_2\right),
	\end{align}
	where the third line follows from \eqref{eq:sufficient decrease} and \eqref{eq:KL for f}, {and the forth line utilizes  \eqref{eq:subdifferential} (i.e., $\vd_k = (2(\vy_{k-1} - \vy_{k}), \vzero) \in \partial f(\vx_k,\vy_k)$) which implies that 
		\[\dist(0,\partial f(\vx_k,\vy_k)) \leq \|\vd_k\| = 2 \|\vy_{k-1} - \vy_{k}\|_2.
		\]}
	Repeating the above equation for $k$ from $k_2$ to $\infty$ and summing them gives
	\e
	\sum_{k=k_2}^\infty \|\vy_{k+1} - \vy_k\|_2 \leq \frac{2C}{1-\theta}\left(f(\vz_{k_2}) - f(\vz^\star)\right)^{1-\theta} + \|\vy_{k_2-1} - \vy_{k_2}\|_2 < \infty, \label{eq:square root summable}
	\ee
	which implies that the series $\{\sum_{k=k_2}^m \|\vy_{k+1} - \vy_k\|_2\}_m$ is convergent. Thus,
	\begin{align*}
	\limsup\limits_{m\rightarrow \infty, m_1,m_2\geq m } ~ \sum_{k=m_1}^{m_2} \left\| \vy_{k+1} - \vy_{k}\right\|_2 = 0.
	\end{align*}
	From triangle inequality we have
	\begin{align*}
	\limsup\limits_{m\rightarrow \infty, m_1,m_2\geq m } ~ \sum_{k=m_1}^{m_2} \left\| \vy_{k+1} - \vy_{k}\right\|_2
	\geq \limsup\limits_{m\rightarrow \infty, m_1,m_2\geq m } ~ \left\| \vy_{m_2+1} - \vy_{m_1}\right\| _2,
	\end{align*}
	which gives that $\limsup_{m\rightarrow \infty, m_1,m_2\geq m } ~ \left\| \vy_{m_2} - \vy_{m_1}\right\|_2 =0$. Thus the sequence $\{\vy_{k}\}$ is Cauchy, hence it is convergent.
	
	Due to $\|\vy_k - \vy_{k-1}\|_2 \rightarrow 0$, there exists $k_1\ge k_0$ such that $\|\vy_{k} - \vy_{k-1}\|_2 \leq \eps$ for all $k\geq k_1$, where $\eps>0$ is a fixed constant defined in local contraction property in \Cref{assum}. It then follows from \eqref{eq: local contraction} that
	\[
	\| \vx_{k+1}- \vx_{k}\|_2 \leq \beta \|\vy_{k} - \vy_{k-1}\|_2, \ \forall \ k\geq \max\{k_0,k_1\}.
	\]
	Now invoking \eqref{eq:square root summable} gives
	\[
	\sum_{k=\max\{k_1,k_2\}}^\infty \|\vx_{k+1} - \vx_k\|_2 \leq \beta \sum_{k=\max\{k_1,k_2\}}^\infty \|\vy_{k} - \vy_{k-1}\|_2 < \infty,
	\]
	which (with a similar argument for $\{\vy_k\}$) implies that $\{\vx_{k}\}$ is convergent.
\end{proof}

\subsection{Convergence rate}
\Cref{thm:main convergence} reveals that the sequence $\{(\vx_k,\vy_k)\}$ is convergent. Given the explicit KL exponent $\theta$ in \Cref{lem:KL}, we can have the convergence rate concerning how fast the sequence $\{(\vx_k,\vy_k)\}$ converges to its limit point. We note that the connection between convergence rate and the KL exponent $\theta$ has been populated and exploited in \cite{attouch2009convergence,attouch2010proximal,bolte2014proximal}. The following result establishes the convergence rate for the sequence $\{(\vx_k,\vy_k)\}$ based on the explicit KL exponent $\theta$.

\begin{thm} (convergence rate) Suppose the sequence $\{\vz_k=(\vx_k,\vy_k)\}$ is generated by \Cref{alg:APM} and converges to a critical point $\vz^\star=(\vx^\star,\vy^\star)$, and assume the function $f$ obeys the KL property with the KL exponent $\theta$ at this critical point $\vz^\star$. Then we have
	\begin{enumerate}[(i)]
		\item if $\theta = 0$, $\{\vz_k\}$ converges to $\vz^\star$ in a finite step.
		\item if $\theta \in (0,\frac{1}{2}]$, then there exist a $0<\rho<1$, $\widetilde{c}>0$ and a positive integer $\widetilde k$ such that
		\e\|\vz_k - \vz^\star\|_2 \leq \widetilde{c}\cdot \rho^k,~~\forall~k\geq \widetilde k.
		\label{eq:rate 1}\ee
		\item if $\theta \in (\frac{1}{2},1)$, then there exist a $\overline{c} >0$ and a positive integer $\overline k$ such that
		\e
		\|\vz_k - \vz^\star\|_2 \leq \overline c \cdot k^{-\frac{1-\theta}{2\theta -1}},~~\forall~k\geq \overline k.
		\label{eq:rate 2}\ee
	\end{enumerate}
	\label{thm:convergence rate}\end{thm}

\begin{proof}[Proof of Theorem \ref{thm:convergence rate}]\footnote{The proof of Theorem \ref{thm:convergence rate} shares similar strategies as those in \cite{attouch2009convergence,attouch2010proximal,bolte2014proximal}. However, as we explained in \Cref{sec:proof highlight}, the sufficient decrease property (and also the safeguard property) in \eqref{eq:sufficient decreae intro} which is utilized here is slightly different than the standard one as in \eqref{eq:sufficient decreae intro classical}. Thus, we include the proof of  Theorem \ref{thm:convergence rate}.}
	Since the function $f$ satisfies the KL property at $z^\star$, there exists $\delta^\star>0$ and $\theta\in[0,1)$ such that
	\[
	|f(\vz) - f(\vz^\star)|^\theta \leq C \dist(0,\partial f(\vz)), \ \forall \ \vz\in \text{B}(\vz^\star, \delta^\star).
	\]
	It follows from the fact $\vz_k\rightarrow \vz^\star$ that there exists a positive integer $k_2$ such that $\|\vz_k -\vz^\star \|_2 \leq \delta^\star$ for all $k\geq k_2$. This together with the above KL property implies
	\e
	|f(\vz_k) - f(\vz^\star)|^\theta \leq C \dist(0,\partial f(\vz_k)),  \ \forall \ k\geq k_2.
	\label{eq:KL consequence}\ee
	In the sequel of the proof, we consider $k\geq k_2$ as we utilize \eqref{eq:KL consequence} to prove the three arguments in Theorem \ref{thm:convergence rate}.
	
	Show $(i)$: In the case where $\theta = 0$, it follows from \eqref{eq:KL consequence} that $\dist(0,\partial f(\vz_k)) \geq {1}/{C}>0$ when $f(\vz_k) >f(\vz^\star)$. 
	%	\begin{align*}
	%	\left\{\begin{aligned}
	%	&\dist(0,\partial f(\vz_k)) \geq {1}/{C},~~~~f(\vz_k) >f(\vz^\star), \\
	%	&\dist(0,\partial f(\vz_k)) =0 ,~~~~f(\vz_k) =f(\vz^\star).
	%	\end{aligned}
	%	\right.
	%	\end{align*}
	Suppose at $k+1$-th iteration $f(\vz_{k+1}) >f(\vz^\star)$, which implies that $\dist(0,\partial f(\vz_{k+1})) \geq 1$. This together with \eqref{eq:sufficient decrease} (i.e. $
	f(\vz_{k})- f(\vz_{k+1}) \geq \alpha\|\vy_{k+1} - \vy_k\|_2^2 $)
	and  \eqref{eq:subdifferential} (i.e. $
	\dist^2(0,\partial f(\vz_{k+1})) \leq 4\|\vy_{k+1} - \vy_k\|_2^2 $) gives
	\[
	f(\vz_{k})- f(\vz_{k+1}) \geq \frac{\alpha}{4} \dist^2(0,\partial f(\vz_{k+1})) \geq \frac{\alpha}{4C^2} >0.
	\]
	Since $\{f(\vz_k)\}$ converges to $f(\vz^\star) \geq 0$, there  exists a finite iteration number $k_3$ such that $f(\vz_{k_3}) = f(\vz^\star) $.

	{
		Show $(ii)$ and $(iii)$: Repeating \eqref{convexity} for all $k$ and summing them up give
		%	\[
		%	\sum_{i = k}^\infty \|\vy_{i+1} - \vy_i\|_2 - \|\vy_{k} - \vy_{k-1}\|_2 \leq  \frac{2C}{1-\theta}[f(\vx_{k},\vy_{k}) - f(\vx^\star,\vy^\star)]^{1-\theta}
		%	\]
		%	hence it follows
		\begin{equation}\label{upper bound on sum}
		\sum_{i = k}^\infty \|\vy_{i+1} - \vy_i\|_2 \leq \|\vy_{k} - \vy_{k-1}\|_2+ \frac{2C}{1-\theta}[f(\vx_{k},\vy_{k}) - f(\vx^\star,\vy^\star)]^{1-\theta}.
		\end{equation}
		The left hand side of the above equation can be further lower bounded as
		\begin{equation}\label{lower bound on sum}
		\sum_{i = k}^\infty \|\vy_{i+1} - \vy_i\|_2 \geq \left\| \sum_{i = k}^\infty \vy_{i+1} - \vy_i \right\|_2 = \| \vy_{k} - \vy^\star\|_2.
		\end{equation}
		%The remaining task is to determine the convergence rate of $\sum_{i = k}^\infty \|\vy_{i+1} - \vy_i\|_2\rightarrow 0$, as $k\rightarrow \infty$,
		Combing \eqref{upper bound on sum} and \eqref{lower bound on sum}  gives
		\begin{align}\label{conv_rate_upperbound}
		\| \vy_{k} - \vy^\star\|_2& \leq \|\vy_{k} - \vy_{k-1}\|_2+ \frac{2C}{1-\theta}[f(\vx_{k},\vy_{k}) - f(\vx^\star,\vy^\star)]^{1-\theta} \nonumber\\
		&\leq \|\vy_{k} - \vy_{k-1}\|_2 + \frac{2C}{1-\theta}\text{dist}(0, \partial f(\vx_k,\vy_k))^{\frac{1-\theta}{\theta}} \nonumber\\
		& = \|\vy_{k} - \vy_{k-1}\|_2 +\frac{4C}{1-\theta} \|\vy_{k} - \vy_{k-1}\|_2 ^{\frac{1-\theta}{\theta}},
		\end{align}
		where the second line is from the KL property of $f(\vz)$ at $\vz^\star$ and the fact that $\frac{1-\theta}{\theta} > 0$, the last inequality follows from \eqref{eq:subdifferential}. Denoting by $Q_k = \sum_{i = k}^\infty \|\vy_{i+1} - \vy_i\|_2$ and noting  that $\|\vy_{k} - \vy_{k-1}\|_2 = Q_{k-1}- Q_{k}$, we have
		\e
		\| \vy_{k} - \vy^\star\|_2 \leq Q_k \leq Q_{k-1}- Q_{k} + \frac{4C}{1-\theta} \left[Q_{k-1}- Q_{k}\right]^{\frac{1-\theta}{\theta}}. \label{eq:convergence rate 1}
		\ee
		Since $Q_{k-1}- Q_{k} \rightarrow 0$,  we define the positive integer $k_4$ such that $Q_{k-1}- Q_{k} < 1,~\forall~k\geq k_4$. We now utilize the proof technique in \cite[Theorem 2]{attouch2009convergence} to obtain the convergence rate for $\| \vy_{k} - \vy^\star\|_2$ from 
		\eqref{eq:convergence rate 1} by considering the following cases:
		\begin{itemize}
			\item  \emph{Case I}: when $\theta \in (0,\frac{1}{2}]$, then there exists  a numerical constant $c_1>0$ such that 
			\e\label{eq:convergence y}
			{\| \vy_{k} - \vy^\star\|_2\le Q_k}    \leq c_1\cdot \rho^k,~~~~\forall ~ k\geq \max\{ k_2,k_4\},
			\ee
			where $\rho = \frac{1+4C/(1-\theta)}{2+ 4C/(1-\theta)} \in (0,1)$ with $C>0$ being the constant in the KL inequality \eqref{eq:KL for f}.    
			
			\item \emph{Case II}: when $\theta \in (\frac{1}{2},1)$, then there exists a constant $c_2>0$ such that
			\e\label{eq:sublinear y}
			{\| \vy_{k} - \vy^\star\|_2\le Q_k} \leq c_2\cdot k^{-\frac{1-\theta}{2\theta -1}},~~~~\forall ~ k> \max\{k_2,k_4\}.
			\ee
			
		\end{itemize}
		We now prove \eqref{eq:rate 1} by using the local contraction property \eqref{eq: local contraction}. Towards that goal, first recall that there exists $k_1$ such that $\|\vy_{k+1} - \vy_k\|_2 \leq \eps$ for all $k\geq k_1$. Thus, the local contraction property \eqref{eq: local contraction} implies that 
		\begin{align*}
		\sum_{i = k}^\infty \|\vy_{i+1} - \vy_i\|_2 &\geq \beta\sum_{i = k+1}^\infty \|\vx_{i+1} - \vx_i\|_2 \\
		&\geq \beta\left\|\sum_{i = k+1}^\infty\vx_{i+1} - \vx_i\right\|_2 =\beta\|\vx_{k+1} - \vx^\star\|_2, \ \forall k \ge k_1,
		\end{align*}
		which together with \eqref{eq:convergence y} implies that 
		\e\label{eq:convergence x}
		\|\vx_{k} - \vx^\star\|_2 \leq \frac{c_1}{\beta}\cdot \rho^k,~~~~\forall ~ k\geq \max\{k_1,k_2,k_4\}.
		\ee
		Setting $\widetilde{k} = \max\{k_1,k_2,k_4\}$ and $\widetilde{c} = \sqrt{2}\max\{c_1, \frac{c_1}{\beta}\}$, we conclude \eqref{eq:rate 1}. Following the same argument, one can obtain \eqref{eq:rate 2}. This completes the proof.
	}	
\end{proof}

{
	\begin{remark}
		When the sets intersect each other and  one of the sets  is H\"{o}lder regular with respect to the other,
		Noll and Rondepierre \cite{noll2016local} also provided sequence convergence and similar convergence rate as \eqref{eq:rate 2} for the alternating projection method. We now discuss the similarities and differences 
		between \cite{noll2016local} and our result in terms of the proof technique.  {On the one hand, both proofs utilize the  three-point property \eqref{eq:3point property} (i.e., the partial sufficient decrease property).}  
		On the other hand, the convergence analysis in \cite{noll2016local} relies on the intersecting property and the H\"{o}lder regularity, while our convergence analysis is based on the local contraction property \eqref{eq: local contraction} and the KL property both of which do not require intersection.
	\end{remark}
	
}

When the sets $\setX$ and $\setY$ are closed and convex (such as subspaces), our framework can recover the linear convergence rate if the interiors of the two sets intersect with each other. This is formally established in the following result\footnote{Similar result has also been established in \cite{bolte2017error} which considered an equivalent form of \eqref{eq:unconstraind}: $\minimize~ f'(\vx)=\|\vx - \calP_{\setY}(\vx)\|^2 + \delta_{\setX}(\vx)$.}.
\begin{cor}\label{cor:linear rate of convex case}
	Suppose that $\setX$ and $\setY$ are closed convex sets and satisfy $\operatorname{reint}(\setX)\cap \operatorname{reint}(\setY) \neq \emptyset$.\footnote{Here, $\operatorname{reint}(\cdot)$ denotes the relative interior of a set.} If the sequence $\{\vz_k=(\vx_k,\vy_k)\}$ generated by \Cref{alg:APM} is bounded,  then it is convergent and converges to a global minimizer of \eqref{eq:min x-y} at a linear rate.
\end{cor}
\begin{proof}[Proof of \Cref{cor:linear rate of convex case}]
	Due to the assumption that $\setX$ and $\setY$ intersect to each other, we denote by $\setI = \setX \cap \setY$ the set of optimal solutions to \eqref{eq:min x-y}.  Since both  $\setX$ and $\setY$ are closed convex sets,  it follows from \eqref{eq:3point convex set} and \eqref{eq:local conctrac convex set} that $\setY$ satisfies the three-point property for and the $\setX$  satisfies the local contraction property. We then invoke \Cref{thm:main convergence} to conclude that $\{\vz_k=(\vx_k,\vy_k)\}$ is convergent and converges to a critical point of \eqref{eq:min x-y}. By convexity, any critical point of \eqref{eq:min x-y} is also a global minimizer. Applying \Cref{thm:convergence rate} which ensures  the linear convergence result when $\theta = \frac{1}{2}$, the remaining task is to show the KL exponent at the limit $\lim_{k\rightarrow\infty}\vz_k = \vz^\star$ for \eqref{eq:min x-y} is indeed $\theta = \frac{1}{2}$.
	% To that end, note in this case the optimization problem \eqref{eq:min x-y} is equivalent to
	% \e \label{eq:min x}
	% \minimize_{\vx\in\R^n} \ \left\{  \frac{1}{2} \dist^2(\vx,\setY)   + \delta_{\setX}(\vx)   \triangleq \widetilde f(\vx) \right\}
	% \ee
	% and which has optimal value equals to 0 provided by any  $\vx^\star\in \setI$ and $\widetilde f(\vx) \geq \frac{1}{2} (\dist^2(\vx,\setX) + \dist^2(\vx,\setY))$, \Cref{alg:APM} can be viewed as a forward backward splitting scheme(i.e. projected gradient method) \cite{combettes2005signal} for solving \eqref{eq:min x}, see \cite{bolte2017error}.
	
	In general, it is not easy to directly compute the KL exponent. A widely used strategy is to connect the KL property with other properties that are much easier to compute. A typical example is called the \emph{error bound} \cite{bolte2017error}: a proper semi-continuous function $h(\vu):\R^d\rightarrow \R$ satisfies a local error bound if thee exist $C_3>0, \theta\in[0,1)$ and $\delta>0$ such that $
	\dist(\vu,\argmin h) \leq C_3 (h(\vu) - \min h)^{\theta}$
	for all $\vu\in\{\vv\in\R^{d}:h(\vv)\leq  \min h + \delta\}$. Here $\argmin$ denotes the set of global minimizers and we assume $\argmin h\neq \emptyset$ (i.e., $h$ achieves its minimum $\min h$).
	
	The following result establishes that for a convex function $h$, its KL property is equivalent to error bound.
	\begin{thm}\cite[Theorem 5]{bolte2017error}Let $h(\vu):\R^d\rightarrow \R$ be proper, convex and semi-continuous. Let $\vu^\star\in \argmin h$.
		\begin{enumerate}[(i)]
			\item (KL inequality implies error bound) If $\left|h(\vu) - h(\vu^\star)\right|^{\theta} \leq C_1 \dist(0, \partial h(\vu))$ holds for all $\vu\in \{\vv\in\R^{d}:h(\vv)\leq  \min h + \delta\}\cap B(\vu^\star,\delta)$, then we have $\dist(\vu,\argmin h) \leq \frac{1}{C_1} (h(\vu) - \min h)^{\theta}$ for all $\vu\in \{\vv\in\R^{d}:h(\vv)\leq  \min h + \delta\}\cap B(\vu^\star,\delta)$.
			
			\item (Error bound implies KL inequality) If $h$ obeys the error bound which is to say $\dist(\vu,\argmin h) \leq C_3 (h(\vu) - \min h)^{\theta}$ for all $\vu\in \{\vv\in\R^{d}:h(\vv)\leq  \min h + \delta\}\cap B(\vu^\star,\delta)$, then $\left|h(\vu) - h(\overline{\vu})\right|^{\theta} \leq \frac{1-\theta}{C_3} \dist(0, \partial h(\vu))$ for all $\vu\in \{\vv\in\R^{d}:h(\vv)\leq  \min h + \delta\}\cap B(\vu^\star,\delta)$.
		\end{enumerate}
		\label{thm:KL to EB}\end{thm}
	\Cref{thm:KL to EB} implies that the KL exponents can be explicitly derived by computing the error bounds.
	Thus, the remaining part is to show the error bound condition for \eqref{eq:unconstraind}.\footnote{The following analysis is inspired from the result in \cite{bolte2017error} which considered an equivalent form of \eqref{eq:unconstraind}: $\minimize~ f'(\vx)=\|\vx - \calP_{\setY}(\vx)\|^2 + \delta_{\setX}(\vx)$. Since our problem \eqref{eq:unconstraind} involves two variables, the analysis is slightly different than the one in \cite{bolte2017error} for $f'(\vx)$. Thus, also for the sake of completion, we include the proof of the error bound condition for \eqref{eq:unconstraind}.}. Towards that end, note that due to the assumption $\operatorname{reint}(\setX)\cap \operatorname{reint}(\setY) \neq \emptyset$, there exist $\overline \vx\in \setI$ and $r>0$ such that $B(\overline \vx,r) \subseteq \setI$. We not define the set of the optimal solutions $\setL = \{\vz:\vx = \vy, \vx\in \setI\}$.
	
	We first consider the case where $\vx\in\setX$ and $\vy\in\setY$. To simplify the following notation, let $d = \|\vx-\vy\|$ and construct
	\e\begin{split}
		\vq  &= \frac{d}{r+d} \overline\vx + \frac{r}{r+d}\vx.
	\end{split}\ee
	Since both $\overline\vx,\vx\in\setX$ and $\setX$ is convex, we have $\vq\in\setX$. On the other hand, we rewrite
	\e
	\vq  = \frac{d}{r+d} \overline\vx + \frac{r}{r+d}\vx = \frac{d}{r+d} (\underbrace{\overline \vx + \frac{r}{d} (\vx - \calP_\setY(\vx))}_{\vp}) + \frac{r}{r+d}\calP_\setY(\vx).
	\label{eq:rewrite q}\ee
	Due to $
	\|\vp - \overline\vx\| = \frac{r}{d}\|\vx - \calP_\setY(\vx)\| \leq \frac{r}{d}d = r,$
	we have $\vp\in\setI$ and hence $\vp\in\setY$. This together with \eqref{eq:rewrite q} indicates that $\vq\in\setY$. Thus, we conclude $\vq\in\setI$. Now, we have
	\begin{align*}
	\dist(\vz,\setL) &\leq \|\vx - \vq\| + \|\vy -\vq\| = \frac{d}{r+d}\|\vx - \overline\vx\| + \left\|\frac{d}{r+d}(\vy -\overline\vx) + \frac{r}{r+d}(\vy - \vx)\right\|\\
	&\leq \frac{d}{r+d}\|\vx - \overline\vx\| + \frac{d}{r+d}\|\vy - \overline\vx\| + \frac{r}{r+d}\left\|\vy - \vx\right\|\\
	& \leq \frac{\|\vx - \overline\vx\| +\|\vy - \overline\vx\| }{r}d + \left\|\vy - \vx\right\|\\
	& \leq \left(\sqrt{2}\frac{\|(\vx,\vy) - (\overline\vx,\overline\vx)\| }{r} +1\right) \left\|\vy - \vx\right\|,
	\end{align*}
	where the last line follows because $d = \|\vx-\vy\|$.
	Let $\overline\vz = (\overline\vx,\overline\vx)$. Finally, we obtain
	\e\begin{split}
		\dist(\vz,\setL) \leq \left(\sqrt{2}\frac{\|\vz_0 - \overline\vz\| }{r} +1\right) \sqrt{f(\vz)}
	\end{split}\label{eq:EB final 1}\ee
	for any $\vz\in B\left(\overline\vz,\dist(\overline\vz,\vz_0)\right)$ and $\vx\in\setX,\vy\in\setY$. Now if $\vx\notin \setX$ or $\vy\notin \setY$, we have $f(\vx,\vy) = \infty$, which implies that \eqref{eq:EB final 1} also holds in this case. Therefore, we conclude the following error bound:
	\[
	\dist(\vz,\setL) \leq \left(\sqrt{2}\frac{\|\vz_0 - \overline\vz\| }{r} +1\right) \sqrt{f(\vz)}
	\]
	for all $\vz\in B\left(\overline\vz,\dist(\overline\vz,\vz_0)\right)$, where $\setL$ is the set of optimal solutions to \eqref{eq:unconstraind}.
\end{proof}

\section{Convergence of alternating projection method for designing structured tight frames}\label{sec:frame design}
As stylized applications of \Cref{thm:main convergence}, we provide new convergence guarantee for designing structured tight frames in \cite{tropp2005designing,tsiligianni2014construction}. An $N\times L$ matrix $\mD = \bmat \vd_1 & \cdots \vd_L \emat$ is said to be a {\em frame} of $\R^N$ if there exist positive real numbers $a$ and $b$ such that $0<a\le b<\infty$ and for each $\vd\in\R^N$,
\begin{align}
a\|\vd\|_2^2 \leq \|\mD^\T \vd\|_2^2 = \sum_{\ell = 1}^L|\langle \vd,\vd_\ell \rangle|^2\le b\|\vd\|_2^2.
\label{eq:frame}\end{align}
It is clear that $L\ge N$ to guarantee that \eqref{eq:frame} holds for any $\vd\in\R^N$. The frame $\mD$ is said {\em overcomplete} if $L>N$. Frame has been widely utilized in signal processing as it provides a redundant and concise way of representing signals, in error detection and correction and the design and analysis of filter banks \cite{tropp2005designing}.

The frame $\mD$ is a {\em tight frame} if $a = b$; that is, the frame $\mD$ satisfies a generalized version of Parseval's identity. Such frames with $a=b$ are said $a$-tight. Clearly, the frame $\mD$ is $a$-tight if and only if it has a singular value decomposition of the form
\e
\mD = \mU\bmat \sqrt{a}\mId & \vzero \emat \mV^\T,
\label{eq:tight frame SVD}\ee
where $\mU\in\R^{N\times N}$ and $\mV\in\R^{L\times L}$ are orthonormal matrices. We define the set contains $a$-tight frames by
\e
\setD_a:=\left\{\mD\in\R^{N\times L}: \mD\mD^\T =a \mId\right\}.
\label{eq:setD alpha}\ee
The following result provides a method to calculate an $a$-tight frame that is closest to an arbitrary matrix in Frobenius norm.
\begin{thm}\cite[Theorem 2]{horn1990matrix} For any $\mZ\in\R^{N\times L}$ with $L\geq N$, let $\mZ = \mU\mSigma\mV^\T$ be its singular value decomposition. A nearest $a$-tight frame to $\mZ$ in Forbenius norm is given by $a\mU\mV^\T$, i.e.
	\[
	a \mU\mV^\T \in \argmin_{\mD\in\setD_a}\|\mD - \mZ\|_F^2.
	\]
	Furthermore, if $\mZ$ has full row-rank, then $\alpha \mU\mV^\T = (\mZ\mZ^\T)^{-1/2}\mZ$ is the unique $\alpha$-tight frame closest to $\mZ$.
\end{thm}

Another equivalent way to characterize the $a$-tight frames is via the eigen decomposition of the corresponding Gram matrices. To be precise, the frame $\mD$ is $a$-tight if and only if its Gram matrix $\mD^\T\mD$ has $N$ nonzero eigenvalues equal $a$, i.e.,
\[
\mD^\T\mD = \mV \bmat a\mId & \mzero \\ \mzero & \mzero \emat \mV^\T, \]
where $\mV$ is an $L\times L$ orthonormal matrix. Define a collection of Gram matrices corresponding to all $N\times L$ $a$-tight frames by
\begin{align}
\setG_a = \bigg\{\mG\in\R^{L\times L}: \mG = \mG^\T, \mG \text{ has eigenvalues} \ (\underbrace{a,\ldots,a}_{N},0,\ldots,0) \bigg\}.
\label{eq:define G}\end{align}
The following result provides a method to calculate a Gram matrix in $\setG_a$ that is closest to an arbitrary matrix in Frobenius norm.
\begin{thm}\cite[Theorem 3]{tropp2005designing} For any $L\times L $ Hermitian $\mZ$, let $\mZ = \mU\mLambda\mU^\T$ be its eigen-decomposition, where the diagonal entries of $\mLambda$ are arranged in the decreasing order. A nearest Gram matrix to $\mZ$ is given by
	\[
	a [\mU]_N[\mU]_N^\T \in \argmin_{\mG\in\setG_a}\|\mG - \mZ\|_F^2,
	\]
	where $[\mU]_N$ is the submatrix of $\mU$ obtained by taking the first $N$ columns of $\mU$.
	Furthermore, if $\lambda_N(\mZ)>\lambda_{N+1}(\mZ)$, then $a [\mU]_N[\mU]_N^\T$ is the unique Gram matrix in $\setG_a$ that is closest to $\mZ$.
\end{thm}

Frames are usually designed based on certain requirements according to different applications. In the following sections, we review several widely utilized structured frames and provide the convergence guarantee of the alternating projection method for designing such frames.

\subsection{Prescribed column norms \cite{tropp2005designing}}\label{sec:prescribed column norms}
As a first illustration example, we consider designing tight frame with prescribed column norms, which has been utilized in the context of constructing optimal signature sequences for DS-CDMA channels~\cite{tropp2005designing}. To that end, we let $\setS$ denote the structural constraint set containing matrices with the prescribed column norms:
\[
\setS :=\left\{\mS\in\R^{N\times L}: \|\vs_\ell\|_2^2= c_\ell,\ \forall \ell\in[L]\right\},
\]
where $c_1,\ldots,c_L$ are the squared column norms of the desired frames and $[L]$ represents the set $\{1,2,\ldots,L\}$. For example, in the DS-CDMA application, the column norms depend on the users¡¯ power constraints~\cite{tropp2005designing}. Let $\calP_{\setS}:\R^{N\times L}\rightarrow \R^{N\times L}$ denote the projection onto the set $\setS$. Similar to \eqref{eq:proj unit sphere}, $\calP_{\setS}$ acts as normalizing each column of the input matrix to the corresponding prescribed column norm:
\e
[\calP_{\setS}(\mZ)](:,n) = \left\{\begin{matrix}c_n \vz_n/\|\vz_n\|_2, & \vz_n\neq \vzero,\\ c_n \vu_n, & \vz_n = \vzero,\end{matrix}\right.
\label{eq:proj S}\ee
where $\vu_n$ represents an arbitrary unit vector.

If the matrix $\mD$ is $a$-tight with prescribed column norms, $a$ can only be $\frac{1}{N}\sum_{\ell=1}^L c_\ell$. To see this, from \eqref{eq:tight frame SVD}, we have
\[
\sum_{\ell=1}^Lc_\ell = \sum_{\ell = 1}^L\|\vd_\ell\|_2^2 = \|\mD\|_F^2 = \|\mU\bmat \sqrt{a}\mId & \vzero \emat \mV^\T\|_F^2 = Na.
\]
Thus, throughout this section, we let $
a = \frac{1}{N}\sum_{\ell=1}^L c_\ell.$
We now design a tight frame with prescribed column norms by solving the following nearest problem:
\e
\minimize_{\mD\in\setD_a,\mS\in\setS} ~g(\mD,\mS) = \|\mD - \mS\|_F^2,
\label{eq:min prescribed norm}
\ee
which can be solved by the alternating projection method (i.e., Algorithm~\ref{alg:APM} by setting $\setX = \setD_a$ and $\setY = \setS$). The convergence analysis for alternating projection method solving \eqref{eq:min prescribed norm} is provided in \cite[Theorem 6]{tropp2005designing}, which guarantees a subsequence convergence. The following result further provides sequence convergence guarantee for alternating projection method designing tight frames with prescribed column norms.
\begin{thm}\label{thm:convergence tight frames column constraint}
	Let $\{(\mD_k,\mS_k)\}\subset \setD_a \times \setS $ be the sequence of iterates generated by \Cref{alg:APM} (by setting $\setX = \setD_a$ and $\setY = \setS$)  for solving \eqref{eq:min prescribed norm} with an initialization $\mS_0$ that has full rank and nonzero columns. Then the sequence $\{(\mD_k,\mS_k)\}$ is convergent and converges to a certain critical point of \eqref{eq:min prescribed norm}.
\end{thm}
\begin{proof}[Proof of \Cref{thm:convergence tight frames column constraint}]
	First note that both $\setD_a$ and $\setS$ are closed semi-algebraic sets. Invoking \Cref{thm:main convergence}, we prove \Cref{thm:convergence tight frames column constraint} by showing that the sequence $\{(\mD_k,\mS_k)\}$ is bounded and  establishing the  three-point property \eqref{eq:3point property} and local contraction property \eqref{eq: local contraction}. Due to the fact that $\mS_k\in\setS$ and $\mD\in\setD_a$, the sequence $\{(\mD_k,\mS_k)\}$ is bounded. Define
	\[
	c_{\min}=\min_{\ell}c_\ell, \ c_{\max} = \max_{\ell}c_\ell.
	\]
	
	%To show the  three-point property \eqref{eq:3point property} and local contraction property \eqref{eq: local contraction}, we require to identify subsets $\overline\setD_a\subset\setD_a$ and $\overline\setS\subset\setS$ such that $(\mD_k,\mS_k)\in \overline\setD_a \times \overline\setS$. Towards that end,
	
	We first review the following useful results which provide  lower bounds on the norm of each column of $\mD_k$ and the smallest singular value of $\mS_k$.
	\begin{prop}\cite[Proposition 23]{tropp2005designing} \label{prop:norm d neq 0} Assume that the initialization $\mS_0\in\setS$ is  full rank. Then, for any $k\geq 1$, we have $(i)$ the Euclidean norm of each column of $\mD_k$ is at least $\frac{c_{\min}}{\sqrt{\sum_{\ell}c_\ell}}$; and $(ii)$ the smallest singular value of $\mS_k$ is at least $\sqrt{c_{\min}}$.
		\label{prop:TF norm}\end{prop}
	As a consequence of \Cref{prop:TF norm}, define the subsets
	\begin{align}
	&\overline \setD_a =\left\{\mD\in\R^{N\times L}: \mD\mD^\T =a \mId, \ \|\vd_\ell\|_2 \geq \frac{c_{\min}}{\sqrt{\sum_{\ell}c_\ell}}, \ \forall \ell\in[L]\right\},\label{eq:bar Da}\\
	&\overline\setS =\left\{\mS\in\R^{N\times L}: \sigma_{\min}(\mS)\geq \sqrt{c_{\min}} , \|\vs_\ell\|_2^2= c_\ell,\ \forall \ell\in[L]\right\}.\label{eq:bar S}
	\end{align}
	In words, \Cref{prop:TF norm} indicates that the sequence $\{\mS_k\}_{k\geq 1}$ lies in $\overline\setS$ (a compact subset of $\setS$ whose elements have full rank), while the sequence $\{\mD_k\}_{k\geq 1}$ lies in $\overline\setD_a$ (a compact subset of $\setD_a$ whose elements have non-zero columns). The following lemma establishes the three-point property for \eqref{eq:min prescribed norm} using \Cref{alg:APM}.
	\begin{lem}\label{lem:3 point property for prescribed  norm}
		(three-point property) %Under the same setup as in \Cref{thm:convergence tight frames column constraint}, we have
		%	\[g(\mD_k,\mS_{k-1}) - g(\mD_k,\mS_k) \geq \frac{c_{\min}}{\revise{c_{\max}}\sqrt{\sum_{\ell}c_\ell}}\|\mS_{k-1}- \mS_k\|_F^2, \ k\geq 1.\]
		{For any $\mS\in\overline\setS$ and $\widetilde \mD \in\overline\setD_a$, we have
			\[
			g(\widetilde \mD,\mS) - g(\widetilde \mD,\widetilde \mS) \geq \frac{c_{\min}}{{c_{\max}}\sqrt{\sum_{\ell}c_\ell}}\|\mS- \widetilde \mS\|_F^2, \ k\geq 1,
			\]
			where $\widetilde \mS = \calP_{\setS}(\widetilde \mD)$ is unique according to the definition of $\calP_{\setS}$ in \eqref{eq:proj S}.
		}		
	\end{lem}
	\begin{proof}[Proof of \Cref{lem:3 point property for prescribed  norm}]
		First note that for any nonzero $\widetilde\vd\in\R^N,\widetilde\vs = c\frac{\widetilde\vd}{\|\widetilde\vd\|_2}$ and any $\vs\in\R^N$ with $\|\vs\|_2 = c$, it follows from \eqref{eq:3point unit sphere} that $
		\|\widetilde \vd - \vs\|_2^2 - \| \widetilde\vd - \widetilde\vs\|_2^2 =\frac{1}{c}\|\widetilde\vd\|_2\|\widetilde\vs - \vs\|_2^2.$
		{Since $\widetilde \mD\in \overline\setD_{\alpha}, \mS\in \overline\setS$ and $\widetilde \mS= \calP_{\setS}(\widetilde \mD)$, we obtain
			\begin{align*}
			&g(\widetilde \mD,\mS) - g(\widetilde \mD,\widetilde \mS) = \|\widetilde \mD- \mS\|_F^2 - \|\widetilde \mD- \widetilde \mS\|_F^2\\
			& \geq \min_\ell \frac{\|\widetilde \mD(:,\ell)\|_2}{c_\ell} \|\widetilde \mS- \mS\|_F^2 \geq \frac{c_{\min}}{c_{\max}\sqrt{\sum_{\ell=1}^Lc_\ell}} \|\widetilde \mS- \mS\|_F^2,
			\end{align*}
			where the last line utilizes \eqref{eq:bar Da} that $\|\widetilde\mD(:,\ell)\|_2 \geq \frac{c_{\min}}{\sqrt{\sum_{\ell=1}^Lc_\ell}}$ for all $\ell\in[L]$.}
	\end{proof}

	On the other hand, the following result establishes the local contraction property.
	\begin{lem}\label{lem:TF norm local conc}
		(local contraction property)
		\begin{comment}Under the same setup as in \Cref{thm:convergence tight frames column constraint}, we have
		\e \label{eq:TF norm local contraction}
		\left\| \calP_{\setD_a}(\mS_{k+1}) - \calP_{\setD_a} (\mS_{k})\right\|_F \leq \frac{1}{N\sqrt{c_{\min}}}\sum_{\ell=1}^L c_\ell \cdot \left\|\mS_{k+1} - \mS_k\right\|_F, \ \forall k\geq 1.
		\ee
		\end{comment}
		{For any $\mS, \widetilde \mS \in \overline\setS$, we have
			\e \label{eq:TF norm local contraction}
			\left\| \calP_{\setD_a}(\widetilde \mS) - \calP_{\setD_a} (\mS)\right\|_F \leq \frac{1}{N\sqrt{c_{\min}}}\sum_{\ell=1}^L c_\ell \cdot \left\|\widetilde \mS - \mS\right\|_F.
			\ee
		}
	\end{lem}
	\begin{proof}[Proof of \Cref{lem:TF norm local conc}]
		We first give out the following useful result.
		\begin{prop} For any $\mA\in\R^{N\times L}, \mB\in\R^{L\times L}$ and $L\times L$ diagonal matrix $\mW$ with positive diagonals $w_1,w_2,\ldots,w_L$, we have $
			\left\| \mA \mW\mB\right\|_F \geq \frac{ \left\| \mA\mB\right\|_F }{\min\limits_{\ell}w_\ell}.$
			\label{prop:AWB to AB}\end{prop}
		\begin{proof}[Proof of \Cref{prop:AWB to AB}]
			First note that
			\begin{align*}
			\left\| \mA \mW\mB\right\|_F^2 = \trace\left(\mA^\T\mA\mW\mB\mB^\T\mW \right) \geq \min\limits_{\ell}w_\ell \cdot \trace\left(\mA^\T\mA\mW\mB\mB^\T \right),
			\end{align*}
			where the last line follows because $\mA^\T\mA\mW\mB\mB^\T $ is a PSD matrix and hence its diagonals are all nonnegative. Similarly,
			\begin{align*}
			&\trace\left(\mA^\T\mA\mW\mB\mB^\T \right)  \geq \min\limits_{\ell}w_\ell \cdot \trace\left(\mA^\T\mA\mB\mB^\T \right) = \min\limits_{\ell}w_\ell \cdot \|\mA\mB\|_F^2.
			\end{align*}
		\end{proof}
		
		{Let $\mS = \mU\mSigma\mV^\T$ and $\widetilde\mS = \widetilde\mU \widetilde\mSigma\widetilde\mV^{\T}$ be the SVD of $\mS$ and $\widetilde\mS$, respectively. Then, we have
			\begin{align*}
			&\| \calP_{\setD_a}(\mS) - \calP_{\setD_a} (\widetilde\mS)\|_F  = \sqrt{a}\left\|\bmat \mU & \widetilde\mU \emat \bmat \mV & -\widetilde\mV \emat^\T\right\|_F  \\ &\leq \sqrt{a} \frac{\left\|\bmat \mU & \widetilde\mU \emat \bmat \mSigma & \\ & \widetilde\mSigma \emat \bmat \mV & -\widetilde\mV \emat^\T\right\|_F}{\min\{\sigma_{\min}(\mS),\sigma_{\min}(\widetilde\mS)\}} \leq \sqrt{a}\frac{\|\mU\mSigma\mV^\T-  \widetilde\mU \widetilde\mSigma\widetilde\mV^{\T}\|_F}{\sqrt{c_{\min}}},
			\end{align*}
		}
		where the first inequality follows from \Cref{prop:AWB to AB} and the last line utilizes  \Cref{prop:TF norm} that the smallest singular value of both $\mS$ and $\widetilde \mS$ is at least $\sqrt{c_{\min}}$. The proof is finished by invoking $a =  \frac{1}{N}\sum_{\ell=1}^L c_\ell$.
	\end{proof}
	This completes the proof for \Cref{thm:convergence tight frames column constraint}.
\end{proof}

\begin{remark}Compared with \cite[Theorem 6]{tropp2005designing} which has the same assumption as \Cref{thm:convergence tight frames column constraint} but  only guarantees the subsequence convergence property of $\{(\mD_k,\mS_k)\}$, \Cref{thm:convergence tight frames column constraint} reveals that the sequence $\{(\mD_k,\mS_k)\}$ generated by the alternating projection method is convergent and converges to a critical point of \eqref{eq:min prescribed norm}. Moreover, once the KL exponent $\theta$ for the objective function of \eqref{eq:min prescribed norm} is available, we can also obtain the convergence rate of the alternating projection method. The KL exponent for quadratic optimization with orthogonality constraints is explicitly given in~\cite{liu2016quadratic,liu2017quadratic}. It is expected that the objective function of \eqref{eq:min prescribed norm} has similar KL component as the ones considered in \cite{liu2016quadratic,liu2017quadratic}.
\end{remark}

\subsection{Equiangular Tight Frames}
As another example, we consider designing tight frame with another important property, the mutual coherence which is defined as
\[
\mu(\mD):=\max_{1\leq i\neq j\leq L} \frac{|\vd_i^\T\vd_j|}{\|\vd_i\|_2\|\vd_j\|_2}
\]
for all $\mD\in\R^{N\times L}$. The mutual coherence $\mu(\mA)$ measures the maximum linear dependency possibly achieved by any two columns of the frame $\mD$. A tight frame with lower mutual coherence has proved to be useful in communication and signal processing, such as sensing matrix design \cite{elad2007optimized,li2013projection,tsiligianni2014construction,li2015designing} and dictionary learning \cite{barchiesi2013learning}.

For any frame $\mD\in\R^{N\times L}$, it is well-known that its mutual coherence is lower bounded by \cite{strohmer2003grassmannian} $
\mu(\mD) \geq  \sqrt{\frac{L-N}{N(L-1)}} =: \xi,$
where the equality holds if and only if $\mD$ is a tight frame and is equiangular, i.e., $
\frac{|\vd_i^\T\vd_j|}{\|\vd_i\|_2\|\vd_j\|_2} = \frac{|\vd_m^\T\vd_n|}{\|\vd_m\|_2\|\vd_n\|_2}$
for all $i\neq j, m\neq n$. With normalized columns, we define an {\em equiangular tight frame} to be a unit-norm tight frame (i.e, each column has unit-norm) in which each pair of vectors has the same absolute inner product.

Equiangular tight frame not only obeys the Parseval's identity property that orthonormal basis has,  but also has equiangular property that orthonormal basis possess (i.e., the inner product between any pari of columns in an orthonormal basis is 0). Though equiangular tight frame has such nice properties, in general, it is not easy to find equiangular tight frames. In particular, equiangular tight frames only exist for rare combinations of $N$ and $L$. For example,  $\mD\in\R^{N\times L}$is an equiangular tight fram only if $L\leq \frac{1}{2}N(N+1)$ \cite{tropp2005designing}.

In \cite{tropp2005designing}, the authors constructed equiangular tight frames using alternating projection method. To briefly mention the main idea,  we note that in an equiangular tight frame, each vector has unit norm and the correlation between any pair of vectors is no larger than $\xi$. Thus, it is easy to first work on the Gram matrix $\mD^\T\mD$ as it displays all of the inner product of the columns within $\mD$. Once we obtain a suitable Gram matrix, it is straightforward to extract the corresponding frame. To that end, define the set of Gram matrices of relaxed equiangular tight frames as
\begin{align}
\setH_{\xi} = \bigg\{\mH\in\R^{L\times L}: \mH = \mH^\T, \diag(\mH) = \vone, \max_{i\neq j}|\mH(i,j)|\leq \xi \bigg\},
\label{eq:define_H}\end{align}
which characterizes the equiangular property. For unit-norm tight frames $\mD$,  we have $\mD$ is $a$-tight with $a =  \frac{1}{N}\sum_{\ell=1}^L c_\ell = L/N$. Thus, throughout this section, we set $a = L/N$. Noting that
the set $\setG_a$ defined in \eqref{eq:define G} characterizes the Parseval's identity property, Tropp et al.~\cite{tropp2005designing} attempted to design equiangular tight frames by solving the following matrix nearest problem
\e \label{eq: min ETF}
\minimize_{\mG\in\setG_a,\mH\in\setH_{\xi}} \|\mG - \mH\|_F^2
\ee
which is addressed by the alternating projection method (i.e., Algorithm \ref{alg:APM} by setting $\setX = \setG_{a}, \setY = \setH_{\xi}$).

\cite[Theorem 9]{tropp2005designing}  provides convergence analysis of alternating projection method solving \eqref{eq: min ETF} and reveals the subsequence convergence property (i.e., any limit point of the iterates is a critical point of \eqref{eq: min ETF}). The following theorem provides new convergence guarantee for designing equiangular tight frames via  alternating projection.
\begin{thm}\label{thm:convergence ETF}
	%	Let $\{(\mG_k, \mH_k)\} \subset \setG_a \times \setH_{\xi} $ be the sequence of iterates generated by \Cref{alg:APM} (by setting $\setX = \setG_{a}, \setY = \setH_{\xi}$)  for solving \eqref{eq: min ETF} with initialization $\mG_0\in\setG_a$ and $\mH_0\in\setH_\xi$ satisfying $\|\mG_0 - \mH_0\|_F^2 \leq {L^2}/(2N^2) -\nu$, where $\nu>0$.   Then the sequence $\{(\mG_k, \mH_k)\}$ is convergent and converges to a  critical point of \eqref{eq: min ETF}.
	
	Let $\{(\mG_k, \mH_k)\} \subset \setG_a \times \setH_{\xi} $ be the sequence of iterates generated by \Cref{alg:APM} (by setting $\setX = \setG_{a}, \setY = \setH_{\xi}$)  for solving \eqref{eq: min ETF}. {Suppose there exist an integer $\widehat k$ and a  constant $\nu>0$ such that $\|\mG_{\widehat k} - \mH_{\widehat k}\|_F^2 \leq {L^2}/(2N^2) -\nu$.}   Then the sequence $\{(\mG_k, \mH_k)\}$ is convergent and converges to a certain critical point of \eqref{eq: min ETF}.
\end{thm}

\begin{proof}[Proof of \Cref{thm:convergence ETF}]
	It is clear that both $\setG_a$ and $\setH_{\xi}$ are compact and semi-algebraic sets, hence the sequence  $\{(\mG_k, \mH_k)\}$ is bounded. According to \Cref{thm:main convergence}, the remaining task is to establish the three-point property \eqref{eq:3point property} and the local contraction property \eqref{eq: local contraction}. To that end, we first give out a useful result characterizing the gap between the $N$-th and $(N+1)$-th eigenvalues of $\mH_k$.% when $k$ is large enough.
	\begin{prop}\label{prop:G H close}
		Under the same setup as in \Cref{thm:convergence ETF}, we have
		\[
		\lambda_N(\mH_k) - \lambda_{N+1}(\mH_k)  \geq \frac{\nu}{a}, \ \forall \ {k\geq \widehat k}.
		\]
	\end{prop}
	
	\begin{proof}[Proof of \Cref{prop:G H close}] Noting that $a = L/N$, it follows from the assumption that $\|\mG_{k} - \mH_{k}\|_F^2 \leq \frac{a^2}{2} - \nu$ for all $k\geq \widehat k$. For any $L\times L$ symmetric matrix $\mA$, suppose its $N$-th and $(N+1)$-th eigenvalues are $\tau$ and $\tau - \varrho$, where $\varrho\geq 0$. Since the $N$-th and $(N+1)$-th eigenvalue of a matrix in $\setG_a$ are $a$ and zero, the Wielandt-Hoffman theorem shows that
		\begin{align*}
		\dist(\mA,\setG_a)^2 &\geq (a - \tau)^2 + (\tau-\varrho)^2 = 2\left(\tau - \frac{a+\varrho}{2}\right)^2 + \frac{a^2}{2} + \frac{\varrho^2}{2} - a \varrho\geq \frac{a^2}{2} - a \varrho,
		\end{align*}
		which together with $\|\mG_{k} - \mH_{k}\|_F^2 \leq \frac{\alpha^2}{2} - \nu$ implies that $
		\varrho \geq \frac{\nu}{a}. $
		Thus, for all $k\geq \widehat k$, the gap between the $N$-th and $(N+1)$-th eigenvalues of $\mH_k$ is at least $\frac{\nu}{a}$.		\end{proof}
	
	{As a consequence of \Cref{prop:TF norm}, define the subset
		\begin{align}
		\overline\setH_\xi =\left\{\mH\in\setH_\xi: \lambda_N(\mH) - \lambda_{N+1}(\mH)  \geq \frac{\nu}{a} \right\}.\label{eq:bar S}
		\end{align}
		In words, \Cref{prop:G H close} indicates that the sequence $\{\mH_k\}_{k\geq \widehat k}$ lies in $\overline\setH_\xi$.
	} {To verify the three-point property and local contraction property in \Cref{assum}, we set $\overline \setG_a = \setG_a$.}
	Now noting that $\setH_\xi$ is a closed convex set, it following from \eqref{eq:3point convex set} that the three-point property holds. {In particular, for any $\widetilde\mG\in\setG_a,\mH\in \setH_{\xi}$, we have
		\[
		g(\widetilde\mG,\mH) - g(\widetilde\mG,\widetilde\mH) \geq \|\mH - \widetilde\mH\|_F^2,	\]
		where $\widetilde\mH= \calP_{\setH_\xi}(\widetilde\mG)$. Note that this is slightly stronger than the three-point property required in \Cref{assum} since the above holds for any $\mH \in \setH_\xi$. The rest of proving \Cref{thm:convergence ETF} is to estiablish the local contraction property.}
	\begin{lem}\label{lem:ETF local concontraction}
		{	(local contraction property) Denote by $\mu =\frac{\nu}{2a}$. Then  there exist $0<\epsilon\leq \mu$ such that  for any $\mH, \widetilde\mH \in\overline \setH_\xi$ with $\|\mH-\widetilde\mH\|_F\leq \epsilon$, the following holds
			\e \label{eq:ETF local contraction}
			\left\| \calP_{\setG_a}(\mH) - \calP_{\setG_a} (\widetilde\mH)\right\|_2 \leq  \frac{\sqrt{2}a}{\mu} \left\|\mH - \widetilde\mH\right\|_F.
			\ee
		}
	\end{lem}
	\begin{proof}[Proof of \Cref{lem:ETF local concontraction}]
		{	
			Let $\mU_{\mH}$ consists of the eigenvectors of $\mH$ corresponding to the eigenvalues $\lambda_1(\mH)\geq \cdots \ge \lambda_L(\mH)$. Similar notation applies to $\mU_{\widetilde\mH}$. For any $\mD\in\R^{L\times L}$, we use $[\mD]_N$ to denote the submatric of $\mD$ obtained by keeping the first $N$ columns of $\mD$. Due to \Cref{prop:G H close}, we have $\lambda_N(\mH) - \lambda_{N+1}(\mH) \geq 2\mu$ and $\lambda_N(\widetilde\mH) - \lambda_{N+1}(\widetilde\mH) \geq 2\mu$. It then follows that
			\begin{align*}
			\calP_{\setG_a}(\mH) = a [\mU_{\mH}]_N [\mU_{\mH}]_N^\T,\ \calP_{\setG_a}(\widetilde\mH) = a [\mU_{\widetilde\mH}]_N [\mU_{\widetilde\mH}]_N^\T.
			\end{align*}
			Choosing $0<\eps \leq \mu$ and by Weyl's inequality we have
			\[
			\lambda_{N}(\widetilde\mH) - \lambda_{N+1}(\mH) \geq \lambda_{N}(\mH)  + \lambda_{\min}(\widetilde\mH-\mH)- \lambda_{N+1}(\mH)
			\geq \mu.
			\]
			Now invoking the perturbation bounds for eigenvectors\cite{wedin1972perturbation}, one has
			\e
			\left\|(\mId - [\mU_{\mH}]_N[\mU_{\mH}]_N^\T)[\mU_{\widetilde\mH}]_N \right\|_F^2 \leq \frac{\|(\mH - \widetilde\mH)[\mU_{\widetilde\mH}]_N\|_F^2}{\mu^2}.
			\label{eq:perturbation eig vector}\ee
			We prove \eqref{eq:ETF local contraction} by connecting it with \eqref{eq:perturbation eig vector}. To that end, note that
			\begin{align*}
			&\left\| \calP_{\setG_a}(\mH) - \calP_{\setG_a} (\widetilde\mH)\right\|_F^2 = \left\|a [\mU_{\mH}]_N [\mU_{\mH}]_N^\T - a [\mU_{\widetilde\mH}]_N [\mU_{\widetilde\mH}]_N^\T\right\|_F^2\\
			& = 2a^2 \left( \left\|[\mU_{\widetilde\mH}]_N\right\|_F^2 - \trace\left( [\mU_{\widetilde\mH}]_N^\T [\mU_{\mH}]_N [\mU_{\mH}]_N^\T [\mU_{\widetilde\mH}]_N\right)\right)\\
			& = 2a^2 \left( \left\|[\mU_{\widetilde\mH}]_N\right\|_F^2 - 2\trace\left( [\mU_{\widetilde\mH}]_N^\T [\mU_{\mH}]_N [\mU_{\mH}]_N^\T [\mU_{\widetilde\mH}]_N\right) \right) \\
			&\quad +2a^2 \left\|[\mU_{\mH}]_N[\mU_{\mH}]_N^T[\mU_{\widetilde\mH}]_N\right\|_F^2\\
			& = 2a^2 \left\|(\mId - [\mU_{\mH}]_N[\mU_{\mH}]_N^\T)[\mU(\widetilde\mH)]_N \right\|_F^2,
			\end{align*}
			where the thrid equation utilizes the fact that \[
			\trace\left( [\mU_{\widetilde\mH}]_N^\T [\mU_{\mH}]_N [\mU_{\mH}]_N^\T [\mU_{\widetilde\mH}]_N\right) = \left\|[\mU_{\mH}]_N[\mU_{\mH}]_N^T[\mU_{\widetilde\mH}]_N\right\|_F^2.
			\]
			By noting that $
			\|(\mH - \widetilde\mH)[\mU_{\widetilde\mH}]_N\|_F^2 \le \|\mH - \widetilde\mH\|_F^2,$
			we finally get \eqref{eq:ETF local contraction}.	
		}
	\end{proof}
	This completes the proof of \Cref{thm:convergence ETF}.
\end{proof}

\begin{comment}
Note that \Cref{thm:convergence ETF} requires a pair of initialization $\mG_0\in\setG_a$ and $\mH_0\in\setH_\xi$ satisfying $\|\mG_0 - \mH_0\|_F^2 < {L^2}/(2N^2)$. On the other hand, if the alternating projection method is started with other initialization (like a pair of randomly picked matrices $\mG_0\in\setG_a$ and $\mH_0\in\setH_\xi$), then this condition $\|\mG_k - \mH_k\|_F^2 < {L^2}/(2N^2)$ can be severed as indicator that the sequence is convergent. As a direct consequence of \Cref{thm:convergence ETF}, this is formally illustrated in the following result.
\begin{cor}\label{cor:convergence ETF}
Let $\{(\mG_k, \mH_k)\} \subset \setG_a \times \setH_{\xi} $ be the sequence of iterates generated by \Cref{alg:APM} (by setting $\setX = \setG_{a}, \setY = \setH_{\xi}$)  for solving \eqref{eq: min ETF}. Suppose there exist an integer $\widehat k$ and a  constant $\nu>0$ such that $\|\mG_{\widehat k} - \mH_{\widehat k}\|_F^2 \leq {L^2}/(2N^2) -\nu$.   Then the sequence $\{(\mG_k, \mH_k)\}$ is convergent and converges to a certain critical point of \eqref{eq: min ETF}.
\end{cor}
\end{comment}

\begin{remark} As also used in \cite[Theorem 9]{tropp2005designing},  the condition $\|\mG_k - \mH_k\|_F^2 < {L^2}/(2N^2)$ can be severed as indicator that the sequence is convergent. \Cref{thm:convergence ETF} improves upon \cite[Theorem 9]{tropp2005designing} in that it shows the sequence of iterates is convergent under the assumption $\|\mG_{\widehat k} - \mH_{\widehat k}\|_F^2< \frac{L^2}{2N^2}$ for some $\widehat k$, while \cite[Theorem 9]{tropp2005designing} only shows the subsequence convergence property of the iterates, i.e., the iterates has at least one convergent subsequence and the limit point of any convergent subsequence is a critical point. Note that the convergence analysis in \cite{tropp2005designing} is based on the assumption that $\calP_{\setG_a}$ has a unique projection, which is not enough for our analysis. For example, for any $\mH$ and $\widetilde \mH$ such that $\calP_{\setG_a}(\mH)$ and $\calP_{\setG_a}(\widetilde \mH)$ are unique, we are not guaranteed that $\|\calP_{\setG_a}(\mH) - \calP_{\setG_a}(\widetilde \mH)\|_F^2$ is upper bounded by $\|\mH - \widetilde \mH\|_F^2$. To be more precise, let $\lambda_N(\mH) = \tau + \eps$, $\lambda_{N+1}(\mH) = \tau - \eps$ and also let $\widetilde \mH$ have the same eigenvalues as $\mH$ and the same eigenvectors as $\mH$ except the $N$-th eigenvector and $(N+1)$-th eigenvector of $\widetilde \mH$ are the $(N+1)$-th eigenvector and $N$-th eigenvector of $\mH$, respectively. Now we have $\|\mH - \widetilde \mH\|_F^2 = 4\eps^2$ which can be arbitrary small when $\eps$ is very small. On the other hand, $\|\calP_{\setG_a}(\mH) - \calP_{\setG_a}(\widetilde \mH)\|_F^2 = 4(\tau + \eps)^2\approx 4\tau^2$ when $\eps$ is very small.  However, \Cref{lem:ETF local concontraction} ensures that when $\|\mH - \widetilde \mH\|_F^2$ is small enough, then $\|\calP_{\setG_a}(\mH) - \calP_{\setG_a}(\widetilde \mH)\|_F^2$ is also very small given that $\calP_{\setG_a}(\mH)$ and $\calP_{\setG_a}(\widetilde \mH)$ are unique.
\end{remark}

\section{Conclusion}
In this paper, we have provided certain conditions for proper, lower semi-continuous and semi-algebraic sets under which the sequence generated by the alternating projection method is convergent and converges to a critical point. In particular, the convergence is guaranteed by  utilizing the Kurdyka-\L{ojasiewicz} (KL) property and the notion of the three-point property and the local contraction property. As a byproduct, we utilized our new analysis framework to get the linear convergence rate of alternating projection onto closed convex sets. Our new analysis framework has also been utilized to ensure the convergence of alternating projection method for designing structured tight frames. Thus, our work supports the growing evidence that the alternating projection method can be useful for engineering applications.

In the process of showing the convergence rate for alternating projection method onto convex sets, we utilized \Cref{thm:KL to EB} that provides a way to compute the KL exponent $\theta$ through the error bound. It would be of interest to provide a similar approach for computing the KL exponent $\theta$ for \eqref{eq:min x-y} involving general nonconvex and nonsmooth sets. A potential approach is to connecting the KL exponent with other properties, like the transversality established in~\cite{drusvyatskiy2015transversality}. In addition, another interesting question would be whether it is possible to extend our analysis framework to general alternating minimizations.

\bibliographystyle{ieeetr}
%\pagenumbering{arabic}
\bibliography{Convergence}

\end{document}

%% file: macro.tex
\usepackage{url}
\usepackage[hidelinks]{hyperref}

\usepackage{amsmath,amsthm,amssymb,amsbsy}
\usepackage{paralist}
\usepackage{xcolor}
\usepackage{color}
\usepackage{graphicx}
\graphicspath{{./figs/}}
\usepackage{algorithm}
\usepackage{algorithmic}
\usepackage{comment}

\usepackage{fancyhdr}
\usepackage{cite}
 \makeatletter \let\cl@chapter\relax \makeatother
\usepackage{cleveref}

\newtheorem{thm}{Theorem}%[section] %(If you want theorem numbered
\newtheorem{cor}{Corollary}%[section]
\newtheorem{prop}{Proposition}%[section]
\newtheorem{defi}{Definition}%[section]

\newtheorem{assum}{Assumption}

\newtheorem{lem}{Lemma}
\newtheorem{remark}{Remark}

\newcommand{\R}{\mathbb{R}}

\newcommand{\N}{\mathbb{N}}

\newcommand{\e}{\begin{equation}}
\newcommand{\ee}{\end{equation}}
\newcommand{\en}{\begin{equation*}}
\newcommand{\een}{\end{equation*}}
\newcommand{\eqn}{\begin{eqnarray}}
\newcommand{\eeqn}{\end{eqnarray}}
\newcommand{\bmat}{\begin{bmatrix}}
\newcommand{\emat}{\end{bmatrix}}
\newcommand{\vct}[1]{\boldsymbol{#1}}
\newcommand{\mtx}[1]{\boldsymbol{#1}}
\newcommand{\T}{\mathrm{T}}

\newcommand{\trace}{\operatorname{trace}}

\newcommand{\diag}{\operatorname{diag}}

\newcommand{\dist}{\operatorname{dist}}

\newcommand{\set}[1]{\mathbb{#1}}
\newcommand{\domain}{\operatorname{dom}}

\DeclareMathOperator*{\minimize}{\text{minimize}}

\DeclareMathOperator*{\argmin}{\text{arg~min}}

\newcommand{\eps}{\epsilon}

\newcommand{\calL}{\mathcal{L}}

\newcommand{\calP}{\mathcal{P}}

\newcommand{\vd}{\vct{d}}

\newcommand{\vp}{\vct{p}}
\newcommand{\vq}{\vct{q}}

\newcommand{\vs}{\vct{s}}

\newcommand{\vu}{\vct{u}}
\newcommand{\vv}{\vct{v}}

\newcommand{\vx}{\vct{x}}
\newcommand{\vy}{\vct{y}}
\newcommand{\vz}{\vct{z}}
\newcommand{\vzero}{\vct{0}}
\newcommand{\vone}{\vct{1}}

\newcommand{\mA}{\mtx{A}}
\newcommand{\mB}{\mtx{B}}

\newcommand{\mD}{\mtx{D}}

\newcommand{\mG}{\mtx{G}}
\newcommand{\mH}{\mtx{H}}

\newcommand{\mS}{\mtx{S}}

\newcommand{\mU}{\mtx{U}}
\newcommand{\mV}{\mtx{V}}
\newcommand{\mW}{\mtx{W}}

\newcommand{\mZ}{\mtx{Z}}

\newcommand{\mLambda}{\mtx{\Lambda}}

\newcommand{\mSigma}{\mtx{\Sigma}}

\newcommand{\mId}{{\bf I}}

\newcommand{\mzero}{{\bf 0}}

\newcommand{\setD}{\set{D}}

\newcommand{\setG}{\set{G}}
\newcommand{\setH}{\set{H}}
\newcommand{\setI}{\set{I}}

\newcommand{\setL}{\set{L}}

\newcommand{\setN}{\set{N}}

\newcommand{\setS}{\set{S}}

\newcommand{\setV}{\set{V}}

\newcommand{\setX}{\set{X}}
\newcommand{\setY}{\set{Y}}

\newcommand{\wt}{\widetilde}

\newcommand{\norm}[2]{\left\| #1 \right\|_{#2}}

\newcommand{\parans}[1]{\left(#1\right)}

\graphicspath{{./figs/}}

\newlength{\imgwidth}
\newboolean{twoColVersion}
\setboolean{twoColVersion}{false}
\newcommand{\twoCol}[2]{\ifthenelse{\boolean{twoColVersion}} {#1} {#2} }

%% file: AlternatingPorjection-arXiv.bbl
\begin{thebibliography}{10}

\bibitem{bauschke1996projection}
H.~H. Bauschke and J.~M. Borwein, ``On projection algorithms for solving convex
  feasibility problems,'' {\em SIAM review}, vol.~38, no.~3, pp.~367--426,
  1996.

\bibitem{youla1982image}
D.~C. Youla and H.~Webb, ``Image restoration by the method of convex
  projections: Part 1---theory,'' {\em IEEE transactions on Medical Imaging},
  vol.~1, no.~2, pp.~81--94, 1982.

\bibitem{combettes1997convex}
P.~L. Combettes, ``Convex set theoretic image recovery by extrapolated
  iterations of parallel subgradient projections,'' {\em IEEE Transactions on
  Image Processing}, vol.~6, no.~4, pp.~493--506, 1997.

\bibitem{bauschke2002phase}
H.~H. Bauschke, P.~L. Combettes, and D.~R. Luke, ``Phase retrieval, error
  reduction algorithm, and fienup variants: A view from convex optimization,''
  {\em Journal of the Optical Society of America A: Optics and Image Science},
  vol.~19, no.~7, pp.~1334--1345, 2002.

\bibitem{byrne2003unified}
C.~Byrne, ``A unified treatment of some iterative algorithms in signal
  processing and image reconstruction,'' {\em Inverse problems}, vol.~20,
  no.~1, p.~103, 2003.

\bibitem{escalante2011alternating}
R.~Escalante and M.~Raydan, {\em Alternating projection methods}.
\newblock SIAM, 2011.

\bibitem{bauschke2013cluster}
H.~H. Bauschke and D.~Noll, ``On cluster points of alternating projections,''
  {\em Serdica Math. J}, vol.~39, pp.~355--364, 2013.

\bibitem{bauschke2014local}
H.~H. Bauschke and D.~Noll, ``On the local convergence of the douglas--rachford
  algorithm,'' {\em Archiv der Mathematik}, vol.~102, no.~6, pp.~589--600,
  2014.

\bibitem{von1950functional}
J.~von Neumann, ``Functional operators. vol. ii. the geometry of orthogonal
  spaces, volume 22 (reprint of 1933 notes) of annals of math,'' {\em Studies.
  Princeton University Press}, 1950.

\bibitem{aronszajn1950theory}
N.~Aronszajn, ``Theory of reproducing kernels,'' {\em Transactions of the
  American mathematical society}, vol.~68, no.~3, pp.~337--404, 1950.

\bibitem{bregman1965method}
L.~M. Bregman, ``The method of successive projections for finding a common
  point of convex sets,'' in {\em Soviet Math. Dokl.}, vol.~6, pp.~688--692,
  1965.

\bibitem{bauschke1993convergence}
H.~H. Bauschke and J.~M. Borwein, ``On the convergence of von neumann's
  alternating projection algorithm for two sets,'' {\em Set-Valued Analysis},
  vol.~1, no.~2, pp.~185--212, 1993.

\bibitem{boyd2003alternating}
S.~Boyd and J.~Dattorro, ``Alternating projections,'' {\em EE392o, Stanford
  University}, 2003.

\bibitem{tropp2005designing}
J.~A. Tropp, I.~S. Dhillon, R.~W. Heath, and T.~Strohmer, ``Designing
  structured tight frames via an alternating projection method,'' {\em IEEE
  Transactions on information theory}, vol.~51, no.~1, pp.~188--209, 2005.

\bibitem{meyer1976sufficient}
R.~R. Meyer, ``Sufficient conditions for the convergence of monotonic
  mathematicalprogramming algorithms,'' {\em Journal of computer and system
  sciences}, vol.~12, no.~1, pp.~108--121, 1976.

\bibitem{lewis2009local}
A.~S. Lewis, D.~R. Luke, and J.~Malick, ``Local linear convergence for
  alternating and averaged nonconvex projections,'' {\em Foundations of
  Computational Mathematics}, vol.~9, no.~4, pp.~485--513, 2009.

\bibitem{drusvyatskiy2015transversality}
D.~Drusvyatskiy, A.~D. Ioffe, and A.~S. Lewis, ``Transversality and alternating
  projections for nonconvex sets,'' {\em Foundations of Computational
  Mathematics}, vol.~15, no.~6, pp.~1637--1651, 2015.

\bibitem{noll2016local}
D.~Noll and A.~Rondepierre, ``On local convergence of the method of alternating
  projections,'' {\em Foundations of Computational Mathematics}, vol.~16,
  no.~2, pp.~425--455, 2016.

\bibitem{douglas1956numerical}
J.~Douglas and H.~H. Rachford, ``On the numerical solution of heat conduction
  problems in two and three space variables,'' {\em Transactions of the
  American mathematical Society}, vol.~82, no.~2, pp.~421--439, 1956.

\bibitem{hesse2013nonconvex}
R.~Hesse and D.~R. Luke, ``Nonconvex notions of regularity and convergence of
  fundamental algorithms for feasibility problems,'' {\em SIAM Journal on
  Optimization}, vol.~23, no.~4, pp.~2397--2419, 2013.

\bibitem{attouch2010proximal}
H.~Attouch, J.~Bolte, P.~Redont, and A.~Soubeyran, ``Proximal alternating
  minimization and projection methods for nonconvex problems: An approach based
  on the {K}urdyka-{{\L}}ojasiewicz inequality,'' {\em Mathematics of
  Operations Research}, vol.~35, no.~2, pp.~438--457, 2010.

\bibitem{bolte2014proximal}
J.~Bolte, S.~Sabach, and M.~Teboulle, ``Proximal alternating linearized
  minimization for nonconvex and nonsmooth problems,'' {\em Math. Program.},
  vol.~146, no.~1-2, pp.~459--494, 2014.

\bibitem{lojasiewicz1963propriete}
S.~Lojasiewicz, ``Une propri{\'e}t{\'e} topologique des sous-ensembles
  analytiques r{\'e}els,'' {\em Les {\'e}quations aux d{\'e}riv{\'e}es
  partielles}, vol.~117, pp.~87--89, 1963.

\bibitem{kurdyka1998gradients}
K.~Kurdyka, ``On gradients of functions definable in o-minimal structures,'' in
  {\em Annales de l'institut Fourier}, vol.~48, pp.~769--784, Chartres:
  L'Institut, 1950-, 1998.

\bibitem{bolte2007lojasiewicz}
J.~Bolte, A.~Daniilidis, and A.~Lewis, ``The {{\L}}ojasiewicz inequality for
  nonsmooth subanalytic functions with applications to subgradient dynamical
  systems,'' {\em SIAM J. Optim.}, vol.~17, no.~4, pp.~1205--1223, 2007.

\bibitem{bolte2007clarke}
J.~Bolte, A.~Daniilidis, A.~Lewis, and M.~Shiota, ``Clarke subgradients of
  stratifiable functions,'' {\em SIAM J. Optim.}, vol.~18, no.~2, pp.~556--572,
  2007.

\bibitem{attouch2009convergence}
H.~Attouch and J.~Bolte, ``On the convergence of the proximal algorithm for
  nonsmooth functions involving analytic features,'' {\em Math. Program.},
  vol.~116, no.~1, pp.~5--16, 2009.

\bibitem{attouch2013convergence}
H.~Attouch, J.~Bolte, and B.~F. Svaiter, ``Convergence of descent methods for
  semi-algebraic and tame problems: proximal algorithms, forward--backward
  splitting, and regularized gauss--seidel methods,'' {\em Math. Program.},
  vol.~137, no.~1-2, pp.~91--129, 2013.

\bibitem{yaghoobi2009parametric}
M.~Yaghoobi, L.~Daudet, and M.~E. Davies, ``Parametric dictionary design for
  sparse coding,'' {\em IEEE Transactions on Signal Processing}, vol.~57,
  no.~12, pp.~4800--4810, 2009.

\bibitem{li2013projection}
G.~Li, Z.~Zhu, D.~Yang, L.~Chang, and H.~Bai, ``On projection matrix
  optimization for compressive sensing systems,'' {\em IEEE Trans. Signal
  Process.}, vol.~61, no.~11, pp.~2887--2898, 2013.

\bibitem{li2015designing}
G.~Li, X.~Li, S.~Li, H.~Bai, Q.~Jiang, and X.~He, ``Designing robust sensing
  matrix for image compression,'' {\em IEEE Trans. on Image Process.}, vol.~24,
  no.~12, pp.~5389--5400, 2015.

\bibitem{elad2007optimized}
M.~Elad, ``Optimized projections for compressed sensing,'' {\em IEEE Trans.
  Signal Process.}, vol.~55, no.~12, pp.~5695--5702, 2007.

\bibitem{tsiligianni2014construction}
E.~V. Tsiligianni, L.~P. Kondi, and A.~K. Katsaggelos, ``Construction of
  incoherent unit norm tight frames with application to compressed sensing,''
  {\em IEEE Transactions on Information Theory}, vol.~60, no.~4,
  pp.~2319--2330, 2014.

\bibitem{absil2005convergence}
P.-A. Absil, R.~Mahony, and B.~Andrews, ``Convergence of the iterates of
  descent methods for analytic cost functions,'' {\em SIAM Journal on
  Optimization}, vol.~16, no.~2, pp.~531--547, 2005.

\bibitem{csisz1984information}
I.~Csisz, G.~Tusn{\'a}dy, {\em et~al.}, ``Information geometry and alternating
  minimization procedures,'' {\em Statistics and decisions}, pp.~205--237,
  1984.

\bibitem{niesen2009adaptive}
U.~Niesen, D.~Shah, and G.~W. Wornell, ``Adaptive alternating minimization
  algorithms,'' {\em IEEE Transactions on Information Theory}, vol.~55, no.~3,
  pp.~1423--1429, 2009.

\bibitem{bolte2017error}
J.~Bolte, T.~P. Nguyen, J.~Peypouquet, and B.~W. Suter, ``From error bounds to
  the complexity of first-order descent methods for convex functions,'' {\em
  Math. Program.}, vol.~165, no.~2, pp.~471--507, 2017.

\bibitem{horn1990matrix}
R.~A. Horn and C.~R. Johnson, {\em Matrix analysis}.
\newblock Cambridge university press, 1990.

\bibitem{liu2016quadratic}
H.~Liu, W.~Wu, and A.~M.-C. So, ``Quadratic optimization with orthogonality
  constraints: Explicit {{\L}}ojasiewicz exponent and linear convergence of
  line-search methods,'' in {\em International Conference on Machine Learning},
  pp.~1158--1167, 2016.

\bibitem{liu2017quadratic}
H.~Liu, A.~M.-C. So, and W.~Wu, ``Quadratic optimization with orthogonality
  constraint: Explicit {{\L}}ojasiewicz exponent and linear convergence of
  retraction-based line-search and stochastic variance-reduced gradient
  methods,'' 2017.

\bibitem{barchiesi2013learning}
D.~Barchiesi and M.~D. Plumbley, ``Learning incoherent dictionaries for sparse
  approximation using iterative projections and rotations,'' {\em IEEE
  Transactions on Signal Processing}, vol.~61, no.~8, pp.~2055--2065, 2013.

\bibitem{strohmer2003grassmannian}
T.~Strohmer and R.~W. Heath, ``Grassmannian frames with applications to coding
  and communication,'' {\em Applied and Computational Harmonic Analysis},
  vol.~14, no.~3, pp.~257--275, 2003.

\bibitem{wedin1972perturbation}
P.-{\AA}. Wedin, ``Perturbation bounds in connection with singular value
  decomposition,'' {\em BIT Numerical Mathematics}, vol.~12, no.~1,
  pp.~99--111, 1972.

\end{thebibliography}
